%% file: main.tex
\pdfoutput=1
\documentclass[onefignum,onetabnum]{siamart220329}

\usepackage[utf8]{inputenc}
\usepackage[T1]{fontenc}
\usepackage{enumerate, enumitem}
\usepackage{lmodern, microtype}
\usepackage{graphicx}
\usepackage{verbatim}
\usepackage{amsmath, amssymb, bm, esint, mathtools, xfrac}
\usepackage{bm, esint, mathtools, xfrac}
\usepackage{mathrsfs}
\usepackage[bb=dsserif]{mathalpha}
\usepackage{booktabs}
\usepackage{colortbl}
\usepackage{tabulary}
\usepackage{diagbox}
\usepackage{caption}
\usepackage{float}
\usepackage[english]{babel}

\usepackage{amsfonts}
\usepackage{epstopdf}
\usepackage{algorithmic}
\ifpdf
\DeclareGraphicsExtensions{.eps,.pdf,.png,.jpg}
\else
\DeclareGraphicsExtensions{.eps}
\fi

\newsiamremark{remark}{Remark}
\newsiamremark{hypothesis}{Hypothesis}
\crefname{hypothesis}{Hypothesis}{Hypotheses}
\newsiamthm{assum}{Assumption}
\newsiamthm{cor}{Corollary}
\newsiamthm{defi}{Definition}
\newsiamthm{prop}{Proposition}
\newsiamthm{thm}{Theorem}

\headers{FE discretization of steady, generalized Navier--Stokes equations}{J. Je\ss{}berger and A. Kaltenbach}

\title{\vspace*{-2mm}Finite element discretization of the steady, generalized Navier--Stokes equations with inhomogeneous Dirichlet boundary conditions
}

\author{Julius Je\ss{}berger\thanks{Department of Applied Mathematics, University of Freiburg, Ernst--Zermelo-Str. 1, D-79104 Freiburg, Germany 
		(\email{julius.jessberger@mathematik.uni-freiburg.de}).}
	\and Alex Kaltenbach\thanks{Institute of Mathematics, Technical University of Berlin, Straße des 17.~Juni 135, 10623, Berlin, Germany
		(\email{kaltenbach@math.tu-berlin.de}).}}

\usepackage{amsopn}

\ifpdf
\hypersetup{
	pdftitle={Finite element discretization of steady, generalized Navier--Stokes equations with inhomogeneous Dirichlet boundary conditions},
	pdfauthor={J. Je\ss{}berger and A. Kaltenbach}
}
\fi

\input{header-commands}

\begin{document}
	
	\maketitle
	\thispagestyle{empty}
	\begin{abstract}
		We propose a finite element discretization for the steady, generalized Navier--Stokes equations for fluids with shear-dependent viscosity, completed with inhomogeneous Dirichlet boundary conditions and an inhomogeneous divergence constraint. We establish (weak) convergence of discrete solutions  as well as a priori error estimates for the velocity vector field and the scalar kinematic pressure. Numerical experiments complement the theoretical findings.\vspace*{-1mm}
	\end{abstract}
	
	\begin{keywords}
		Generalized Newtonian fluids; finite element method; weak convergence; a priori~error estimates; inhomogeneous Dirichlet boundary condition; velocity; pressure.\vspace*{-1mm}
	\end{keywords}
	
	\begin{MSCcodes}
		35J60; 35Q35; 65N12; 65N15; 65N30; 76A05.\vspace*{-1mm}
	\end{MSCcodes}
	
	\section{\hspace{-0.25em}Introduction}
	
	\hspace{-0.25em}The steady motion of a homogeneous, incompressible~\mbox{generalized} Newtonian fluid can be modelled by the \textit{generalized Navier--Stokes equations}
	\begin{align} \label{eq:fem:main_problem1}
		- \textup{div}\, \vec{S}(\vec{Dv}) + \textup{div}\,(\vec{v}\otimes\vec{v}) + \nabla q = \vec{f} \quad \textrm{ in } \Omega\,.
	\end{align}
	Here, $\Omega\subseteq \mathbb{R}^d$, $d\in \{2,3\}$, denotes a bounded polyhedral Lipschitz domain, which is occupied by the fluid, and $\vec{f}\colon \Omega\to \mathbb{R}^d$ denotes an external body force. The \textit{velocity vector field} 
	$\vec{v}\colon \overline{\Omega}\to \mathbb{R}^d$ and the scalar \textit{kinematic pressure} $q\colon \Omega\to \mathbb{R}$~are~the~unknowns. Moreover,
	we assume that the viscosity can be expressed as a function of the~\textit{shear~rate}. Therefore, we consider an \textit{extra stress tensor} $\vec{S}\colon \mathbb{R}^{d\times d}\to \mathbb{R}_{\textup{sym}}^{d\times d}$~that~has~\textit{$(p,\delta)$-structure} (cf. Assumption \ref{assum:extra_stress}). While the particular case $p=2$ results in the well-known~Navier--Stokes equations for Newtonian fluids, the cases $p<2$ and $p>2$ correspond to \textit{shear-thinning} and \textit{shear-thickening} fluids, respectively.\enlargethispage{2mm}
	
	The system is completed by the \textit{inhomogeneous Dirichlet boundary condition}
	\begin{alignat}{2} 
		\vec{v} &= \vec{g}_2 &&\quad \textrm{ on } \del\Omega\,,\label{eq:fem:main_problem2}
		\intertext{
			and, for the sake of generality, the \textit{inhomogeneous divergence constraint}
		}
		\div \vec{v} &= g_1&& \quad \textrm{ in } \Omega\,.\label{eq:fem:main_problem3}
	\end{alignat}
	
	Several results have been established for finite element (FE) discretizations~of~the corresponding generalized Stokes system: in~\cite{BBDR12_fem}, error estimates are proved for homogeneous Dirichlet boundary conditions. Similar results have been established in~\cite{lan12_fem}~for~fluids with shear-rate- and pressure-dependent viscosity, $p < 2$, and mixed inhomogeneous boundary conditions.
	To the best of the authors' knowledge, FE discretizations for the system \eqref{eq:fem:main_problem1}--\eqref{eq:fem:main_problem3} have solely been investigated in the case of homogeneous Dirichlet boundary conditions as well as merely (weak) convergence results has been established. This was done in a fairly general context, e.g., for implicitly constituted fluids and unsteady flows thereof (cf.\ \cite{DKS13a,Tscherpel_phd}).
	In a recent series of articles, a local discontinuous Galerkin (LDG) scheme for the system \eqref{eq:fem:main_problem1}--\eqref{eq:fem:main_problem3} has been proposed in the case~$p > 2$.\linebreak Existence and (weak) convergence  results  have been derived (cf.\ \cite{KR23_ldg1}) as well as a priori error estimates in the case of homogeneous Dirichlet boundary~conditions~(cf.~\cite{KR23_ldg2, KR23_ldg3}).
	
	In this work, we provide a FE discretization for the fully inhomogeneous, steady system \eqref{eq:fem:main_problem1}--\eqref{eq:fem:main_problem3}. Then, for this FE discretization, we establish its well-posedness (i.e., existence of discrete solutions), stability (i.e., a priori energy bounds),  convergence (i.e., (weak) convergence towards a solution of an appropriate weak formulation of  the system \eqref{eq:fem:main_problem1}--\eqref{eq:fem:main_problem3}), as well as a priori error estimates.\enlargethispage{8mm}
	
	While most of our assumptions (like existence of a Fortin interpolation~\mbox{operator}) are rather uncritical, we require certain smallness constraints for the following~\mbox{reasons}:
	in the shear-thickening and in the homogeneous case, the corresponding operator is coercive and, thus, existence of weak solutions can be proved with the theory~of~pseudo-monotone operators (and, additionally, the Lipschitz truncation method~if~$p \leq \tfrac{3d}{d+2}$) (cf.\ \cite{r-mol-inhomo, lions-quel, fms2, dms}). Furthermore, the existence of weak solutions is only known under additional assumptions like smallness of the data $g_1$ and $\vec{g_2}$ (cf.\ \cite{Leray33, Hopf41, Galdi, KPR_2d3d, mikelic, JR21_inhom}). This transfers to the discrete situation: if $p \leq 2$, then we require smallness of the data $g_1$~and~$\vec{g_2}$, in order to show  the existence and (weak) convergence~of~discrete~solutions.
	In addition, because of quadratic terms coming from the convective term $ \textup{div}\,(\vec{v}\otimes\vec{v}) $,\linebreak the error estimate for the velocity requires smallness of the solution $\vec{v}$ in any case. To the best of the authors' knowledge, this corresponds to the state of the art regarding error estimates for the (much more intensively researched) classical Navier--Stokes~equations.
	
	\textit{This paper is structured as follows:} In Section \ref{sec:preliminaries}, we introduce the needed~notation, the finite element setting and precise formulations of continuous and~discrete~problem. Then, Section~\ref{sec:fem:ex_conv} is dedicated to the 
	well-posedness, stability and (weak) convergence of the discrete formulation.  In Section \ref{sec:error}, we derive a priori error estimates~for~the~velocity vector field and the kinematic pressure. In~Section~\ref{sec:fem:num_experiments},
	numerical experiments are carried out to complement the theoretical findings.
	
	\input{section1} 
	
	\input{section2} 
	
	\input{section3} 
	
	\input{section4} 
	
	\section{Conclusion}
	For the steady, generalized Navier-Stokes system with inhomogeneous Dirichlet boundary data and an inhomogeneous divergence constraint \eqref{eq:fem:main_problem1}--\eqref{eq:fem:main_problem3}, we proposed a FE discretization. We established the well-posedness, stability and (weak) convergence to a solution of the continuous problem under conditions that are similar to the existence theory for the continuous problem with inhomogeneous~boundary conditions.

	Then, we derived a priori error estimates for this discretization, for the velocity vector field and the kinematic pressure. The findings match the state of the art for the generalized Stokes system (cf.\ \cite{BBDR12_fem}) and for an LDG discretization (cf.\ \cite{KR23_ldg2,KR23_ldg3}). A new approach was followed via considering the $L^{\smash{\ell'}}(\Omega)$-norm, where $\ell\coloneqq \max \{2, s \}$,~if~$p<2$. It allows for a linear error decay rate in contrast to previous results in the $L^{s'}(\Omega)$-norm.
	
	Numerical experiments with critical regularity have been conducted as benchmark tests complementing the theoretical findings in Section~\ref{sec:error}. These experiments confirm the quasi-optimality of the obtained a priori error estimates for the velocity vector field. The same holds for the a priori error estimates  for the pressure in the $L^{p'}(\Omega)$-norm if $p \in [{\tfrac{3d}{d+2}},2]$.
	If $p<2$, the linear error decay rates for the pressure error in the weaker $L^{\smash{\ell'}}(\Omega)$-norm seem to be quasi-optimal at least if $\ell = 2$, this is if $p \geq \smash{\tfrac{4d}{d+4}}$.~Unfortunately, the case $p \in (\smash{\tfrac{2d}{d+1}}, \smash{\tfrac{4d}{d+4}})$ could not be investigated in our setup since this interval is empty if $d=2$. Conducting respective studies for $d=3$ will be part of future research. 
	Nevertheless, the results hint that the $L^{\smash{\ell'}}(\Omega)$-norm is a suitable error measure for the pressure if $p < 2$.
	The sub-optimality of the results for the pressure if $p>2$ is similar to comparable studies for the generalized Stokes~system~\cite{BBDR12_fem}~and~an~LDG~discretization~\cite{KR23_ldg3}.
	
	\appendix
	\section{Discrete Lipschitz truncation}
	
	\hspace*{-0.1em}For the convenience~of~the reader, we recall a result on the discrete Lipschitz truncation and its divergence-correc\-ted version which is used in the convergence proof. The statements are due to \cite{DKS13a}; see also \cite[Lem. 2.29, 2.30]{Tscherpel_phd}. Due to the approximability property~\cite[Thm. 4.6]{DR07_interpolation} and density of smooth functions,  our assumptions on the projection operators~are~sufficient.
	\begin{theorem}[Discrete Lipschitz truncation] \label{thm:fem:dlt}
		Let Assumptions~\ref{assum:fem:geometry}, \ref{assum:fem:projection_x}, \ref{assum:fem:projection_y} and \ref{assum:fem:localbasis} be satisfied, let $p \in (1, \infty)$, and let $n \in \N$ indicate a sequence of triangulations with mesh width $h_n \to 0$. Assume that the sequence $\vec{v}^n \in V_{h_n}$ converges weakly to zero in $W_0^{1,p}(\Omega)$ as $n \to \infty$.
		Then, for every $n, j \in \N$, there exist\vspace*{-0.25mm}
		\begin{itemize}
			\item real numbers $\lambda_{n,j} \in [ 2^{2^j}, 2^{2^{j+1}-1} ]$;
			\item sets $B_{n,j} = \operatorname{int} ( \{K \in \mathcal{T}_{h_n,j}\} ) \subseteq \Omega$, where $\mathcal{T}_{h_n,j} \subseteq \mathcal{T}_{h_n}$ is a subset~of~elements;
			\item truncated functions $\vec{v}^{n,j}\in W_0^{1,\infty}(\Omega) \cap V_{h_n}$;
		\end{itemize}
		with the following properties:
		\begin{itemize}
			\item[(A.1)] \hypertarget{eq:fem:lteq}{} $\vec{v}^{n,j}= \vec{v}^n$ on $ \Omega \setminus B_{n,j}$ for all $n,j\in \mathbb{N}$;
			\item[(A.2)] \hypertarget{eq:fem:ltb}{} $\|\lambda_{n,j} \chi_{B_{n,j}}\|_p \leq c\, 2^{-j/p} $ for all $n,j\in \mathbb{N}$;
			\item[(A.3)] \hypertarget{eq:fem:ltw1inf}{} $	\|\nabla \vec{v}^{n,j}\|_{\infty} \leq c \,\lambda_{n,j}$ for all $n,j\in \mathbb{N}$;
			\item[(A.4)] $\vec{v}^{n,j}\to \vec{0} $ in $L^{\infty}(\Omega)$ $(n\to \infty)$ for all $j\in \mathbb{N}$;
			\item[(A.5)]  $\nabla\vec{v}^{n,j}\weakstarto \vec{0} $ in $L^{\infty}(\Omega)$ $ (n\to \infty)$ for all $j\in \mathbb{N}$.
		\end{itemize}
	\end{theorem}
	
	\begin{lemma}[Discrete divergence-corrected Lipschitz truncation] \label{lem:fem:dlt}
		Let the assumptions of Theorem~\ref{thm:fem:dlt} be satisfied and suppose, in addition, that  $\vec{v}^n \in V_{h_n, 0}$~for~all~$n\in \mathbb{N}$. Let $\tilde{\vec{v}^{n,j}} \in V_{h_n}$ be the discrete Lipschitz truncations from Theorem~\ref{thm:fem:dlt}. Then, there is a double sequence $\vec{v}^{n,j}\in V_{h_n,0}$, s.t.\vspace*{-0.5mm}
		\begin{itemize}
			\item[(A.6)]  \hypertarget{eq:fem:ltdiff}{} $\|\vec{v}^{n-j} - \tilde{\vec{v}^{n,j}}\|_{1,p} \leq c \,2^{-j/p} $ for all $ n, j \in \N$;
			\item[(A.7)]  \hypertarget{eq:fem:ltls}{} $	\vec{v}^{n,j}\to \vec{0} $ in $ L^s(\Omega)$ $( n\to \infty)$ for all $j\in \mathbb{N}$;
			\item[(A.8)]  \hypertarget{eq:fem:ltw1s}{} $	\vec{v}^{n,j}\weakto \vec{0} $ in $ W_0^{1,s}(\Omega) $ $(n \to \infty) $ for all $j\in \mathbb{N}$ and $s \in (1, \infty)$.\enlargethispage{13mm}
		\end{itemize}
	\end{lemma}
	
	\section*{Acknowledgments}
	Alex Kaltenbach acknowledges support from the Deutsche Forschungsgemeinschaft  (DFG, German Research
	Foundation) ---  within the Walter--Benjamin-Program (project number: 525389262) as well as the University of Pisa for their hospitality.
	

\end{document}

%% file: header-commands.tex
\renewcommand{\phi}{\varphi}			
\renewcommand{\vec}[1]{\boldsymbol{#1}}	
\renewcommand{\tilde}[1]{\widetilde{#1}}

\newcommand{\N}{\mathbb{N}}				
\newcommand{\R}{\mathbb{R}}				

\newcommand{\abs}[1]{\left\lvert#1\right\rvert}	  	
\renewcommand{\d}{^{\ast}}							
\newcommand{\del}{\partial}							
\providecommand{\meantmp}[2]{#1\langle{#2}#1\rangle}
\providecommand{\mean}[1]{\meantmp{}{#1}}			
\newcommand{\norm}[1]{\left\lVert#1\right\rVert}  	
\newcommand{\supp}{\operatorname{supp}}				
\newcommand{\weakstarto}{\overset{*}{\rightharpoonup}}		
\newcommand{\weakto}{\rightharpoonup}						
\newcommand{\spann}{\operatorname{span}}				

\newcommand{\Bog}{\mathcal{B}}						
\renewcommand{\div}{\operatorname{div}}				

\newcommand{\D}[1]{\vec{D #1}}						
\newcommand{\tr}{\operatorname{tr}}					

\newcommand{\dx}{\, dx}


%% file: section1.tex
\section{Preliminaries}\label{sec:preliminaries}

In this section, we provide the needed notation for the analysis and give precise definitions of the continuous problem and a finite~\mbox{element}~\mbox{discretization}.\enlargethispage{2mm}

\subsection{Notation}

We employ $c,C>0$ to denote generic constants,~that~may change from line to line, but do not depend on the crucial quantities. Moreover, we~write~$u\sim v$ if and only if there exist constants $c,C>0$ such that $c\, u \leq v\leq C\, u$.

For $k\in \N$ and $p\in [1,\infty]$, we employ the customary
Lebesgue spaces $(L^p(\Omega), \|\cdot\|_p) $ and Sobolev
spaces $(W^{k,p}(\Omega),\|\cdot\|_{k,p})$, where $\Omega
\subseteq \R^d$, $d\in \{2,3\}$, is a bounded,~polyhedral Lipschitz domain. The space $\smash{W^{1,p}_0(\Omega)}$
is defined as those functions from $W^{1,p}(\Omega)$ whose traces vanish on $\partial\Omega$. We equip
$\smash{W^{1,p}_0(\Omega)}$ with the norm $\smash{\|\nabla\cdot\|_p}$. 
The conjugate exponent is denoted by $p' = \tfrac{p}{p-1} \in [1, \infty]$, whereas $p\d$ is the critical Sobolev~exponent, i.e., if $p > d$, then $p\d = \infty$, while $p\d$ is used as a placeholder~for~any~finite~${q \in [1, \infty)}$~if~$p=d$.

We do not distinguish between function spaces for scalar,
vector-~or~\mbox{tensor-valued} functions. However, we denote
vector-valued functions by boldface letters~and~tensor-valued
functions by capital boldface letters. The standard scalar product
between~two vectors $\vec{a} =(a_1,\dots,a_d)^\top,\vec{b} =(b_1,\dots,b_d)^\top\in \mathbb{R}^d$ is denoted by 
$\vec{a} \cdot\vec{b}\coloneqq \sum_{i=1}^d{a_ib_i}$, while the
Frobenius scalar product between two tensors $\vec{A}=(A_{ij})_{i,j\in \{1,\dots,d\}},\vec{B}=(B_{ij})_{i,j\in \{1,\dots,d\}}\in \mathbb{R}^{d\times d }$ is denoted by
$\vec{A}: \vec{B}\coloneqq \sum_{i,j=1}^d{A_{ij}B_{ij}}$. For a tensor $\vec{A}\in \R^{d\times d}$, we denote its symmetric part by $\vec{A}^{\textup{sym}}\coloneqq \frac{1}{2}(\vec{A}+\vec{A}^\top)\in \R^{d\times d}_{\textup{sym}}\coloneqq \{\vec{A}\in \R^{d\times d}\mid \vec{A}=\vec{A}^\top\}$.
Moreover, for a (Lebesgue) measurable set $M\subseteq \mathbb{R}^d$  and (Lebesgue) measurable functions $u,v\in L^0(M)$, we employ the product $(u,v)_M \coloneqq \int_M u v\,\mathrm{d}x$, whenever the right-hand side is well-defined.
The mean value of an integrable function $u\in L^1(M)$ over a (Lebesgue) measurable set $M\subseteq \mathbb{R}^d$ is denoted by ${\mean{u}_M \coloneqq \smash{\fint_M u \,\dx}\coloneqq \smash{\frac 1 {|M|}\int_M u \,\dx}}$.

For $r \in (1, \infty)$, we abbreviate the function spaces\vspace*{-1mm}
\begin{align}
	\begin{aligned}
		X^r &\coloneqq  W^{1,r}(\Omega)\,, &&\quad
		V^r \coloneqq  W_0^{1,r}(\Omega)\,, \\
		Y^r &\coloneqq \smash{ L^{r'}(\Omega)}\,, &&\quad
		Q^r \coloneqq  \smash{L_0^{r'}(\Omega)} \coloneqq  \{ z \in \smash{L^{r'}(\Omega) }\mid \langle z \rangle_{\Omega} = 0 \}\,,
	\end{aligned}
\end{align}
as well as  $X\coloneqq X^p$, $V\coloneqq V^p$, $Y\coloneqq Y^p$, and $Q\coloneqq Q^p$.

%

\subsection{N-functions and Orlicz spaces}

A real convex function $\psi \colon \R^{\geq 0} \to \R^{\geq 0}$ is called \textit{N-function} if it holds that  $\psi(0)=0$,
$\psi(t)>0$~for~all~${t>0}$, $\lim_{t\rightarrow0} \psi(t)/t=0$, and
$\lim_{t\rightarrow\infty} \psi(t)/t=\infty$. 
A Carath\'eodory function $\psi \colon \Omega \times \R^{\geq 0} \to \R^{\geq 0}$ such that $\psi(x,\cdot)$ is an N-function~for~a.e.~${x \in \Omega}$ is called \textit{generalized N-function}.
We~define the (convex) conjugate N-function $\psi^*\colon \hspace*{-0.1em}\R^{\geq 0} \hspace*{-0.1em}\to\hspace*{-0.1em} \R^{\geq 0}$ by
${\psi^*(t)\hspace*{-0.1em} \coloneqq \hspace*{-0.1em}\sup_{s \geq 0} (st - \psi(s))}$~for~all~${t \hspace*{-0.1em}\geq \hspace*{-0.1em} 0}$,~which satisfies
$(\psi^*)' =  (\psi')^{-1}$. A N-function
$\psi$ satisfies the \textit{$\Delta_2$-condition} (in short,
$\psi \in \Delta_2$), if there exists
$K> 2$ such that ${\psi(2\,t) \leq K\,\psi(t)}$ for all
$t \geq 0$. We denote the smallest such constant by
$\Delta_2(\psi) > 0$. We need the following version of the \textit{$\epsilon$-Young inequality}: for every
${\epsilon> 0}$,~there~exists a constant $c_\epsilon>0 $,
depending only on $\Delta_2(\psi),\Delta_2( \psi ^*)<\infty$, such
that~for~every~${s,t\geq 0}$,~it~holds that
\begin{align} \label{eq:fem:orliczyoung}
	s\,t&\leq c_\epsilon \,\psi^*(s)+ \epsilon \, \psi(t)\,.
\end{align}

For $p \in (1,\infty)$ and~$\delta\geq 0$, we define a \textit{special N-function} $\phi \coloneqq  \phi_{p,\delta}\colon\smash{\R^{\geq 0}\to \R^{\geq 0}}$~by
\begin{align} \label{eq:fem:nfunction} 
	\phi(t)\coloneqq  \int _0^t \phi'(s)\, \mathrm{d}s\,, \quad\text{where}\quad
	\phi'(t) \coloneqq  (\delta +t)^{p-2} t\,,\quad\textup{ for all }t\geq 0\,.
\end{align}
The special N-function  satisfies the $\Delta_2$-condition with $\Delta_2(\phi) \leq c\, 2^{\max \{2,p\}}$. 
In addition, the conjugate function $\phi^*$ satisfies the
$\Delta_2$-condition with $\Delta_2(\phi^*) \leq c\,2^{\max \{2,p'\}}$~and ${\phi^*(t) \sim (\delta^{p-1} + t)^{p'-2} t^2}$ uniformly in $t,\delta\geq 0$. 

An important tool in our analysis are shifted N-functions
$\{\psi_a\}_{\smash{a \geq 0}}$  (cf.~\cite{DK08_adaptive,GNSE_DR07}). For a given N-function $\psi\colon\hspace*{-0.1em}\R^{\geq 0}\hspace*{-0.1em}\to\hspace*{-0.1em} \R^{\geq 0}$, we define \textit{shifted N-functions} ${\psi_a\colon\hspace*{-0.1em}\R^{\geq 0}\hspace*{-0.1em}\to\hspace*{-0.1em} \R^{\geq 0}}$,~${a \hspace*{-0.1em}\geq\hspace*{-0.1em} 0}$, for every $a\hspace*{-0.1em}\geq \hspace*{-0.1em} 0$, via
\begin{align*} \label{eq:fem:phi_shifted}
	\psi_a(t)\coloneqq  \int _0^t \psi_a'(s)\,  \mathrm{d}s\,,\quad\text{where }\quad
	\psi'_a(t)\coloneqq \psi'(a+t)\frac {t}{a+t}\,,\quad\textup{ for all }t\geq 0\,.
\end{align*}

Let $M\subseteq \mathbb{R}^d$, $d\in \mathbb{N}$, be a (Lebesgue)~measurable set. For a (Lebesgue)~\mbox{measurable} function $u\hspace*{-0.1em}\in\hspace*{-0.1em} L^0(M)$, the corresponding \textit{modular} is defined  by
$ \rho_{\psi,M}(u)\hspace*{-0.1em}\coloneqq \hspace*{-0.1em}\int_M
{\psi(\abs{u})\,\textup{d}x} $~if $\psi$ is an N-function and
$ \rho_{\psi,M}(u)\hspace*{-0.1em}\coloneqq \hspace*{-0.1em} \int_M{
	\psi(\cdot,\abs{u})\,\textup{d}x} $, if $\psi$
is a generalized~N-function.~Then, for a (generalized)
N-function~$\psi$, we denote by ${L^{\psi}(M)\coloneqq \{u\in L^0(M)\mid
	\rho_{\psi,M}(u)<\infty\}}$, the \textit{(generalized) Orlicz space}. Equipped with
the induced  \textit{Luxembourg norm},  i.e., $\smash{\norm {u}_{\psi,M}\hspace*{-0.1em}\coloneqq\hspace*{-0.1em}  \inf \{\lambda >0\mid \rho_{\psi,M}(u/\lambda)\leq 1\}}$, the (generalized) Orlicz space
$L^\psi(M)$~is~a Banach \hspace*{-0.1mm}space.  \!If $\psi$ \hspace*{-0.1mm}is \hspace*{-0.1mm}a \hspace*{-0.1mm}generalized
N-function, \hspace*{-0.1mm}then, \hspace*{-0.1mm}for \hspace*{-0.1mm}every \hspace*{-0.1mm}$u\!\in \! L^{\psi}(M)$~\hspace*{-0.1mm}and~\hspace*{-0.1mm}${v\! \in\!   L^{\psi^*}(M)}$, there holds the  \textit{generalized Hölder inequality}
\begin{align}\label{eq:gen_hoelder}
	(u,v)_M\leq 2\,\|u\|_{\psi,M}\|v\|_{\psi^*,M}\,.
\end{align}

\subsection{The extra stress tensor and related functions} In this subsection, we specify the precise assumptions on the extra stress tensor and important consequences.

\begin{assum}[extra stress tensor] \label{assum:extra_stress}
	We assume that the extra stress tensor $ \vec{S}\hspace*{-0.1em}\colon\hspace*{-0.15em} \R^{d \times d} \hspace*{-0.15em}\to\hspace*{-0.15em} \R^{d \times d}_{\textup{sym}} $ belongs to $C^0(\R^{d \times d},\R^{d \times d}_{\textup{sym}} ) $, 
	satisfies $\vec{S} (\vec{A}) \hspace*{-0.15em}=\hspace*{-0.15em} \vec{S} (\vec{A}^{\textup{sym}})$~for~all~${\vec{A}\hspace*{-0.15em}\in \hspace*{-0.15em}\R^{d \times d}}$, and $\vec{S}(\mathbf 0)=\mathbf 0$. 
	
	Moreover, we assume that the tensor $\vec{S} = (S_{ij})_{i,j=1,\dots,d}$ has \textup{$(p,\delta)$-structure}, i.e.,
	for some $p \in (1, \infty)$, $ \delta\in [0,\infty)$, and the
	N-function $\phi=\phi_{p,\delta}$ defined in \eqref{eq:fem:nfunction},~there exist constants $C_0, C_1 > 0$ such that for every $\vec{A},\vec{B} \in \R^{d\times d}$, it holds that
	\begin{align}
		({\vec{S}}(\vec{A}) - {\vec{S}}(\vec{B})) : (\vec{A}-\vec{B}) &\ge C_0 \,\phi_{\vert \vec{A}^{\textup{sym}}\vert}(\vert\vec{A}^{\textup{sym}} -
		\vec{B}^{\textup{sym}}\vert) \,,\label{assum:extra_stress.1}
		\\
		\vert \vec{S}(\vec{A}) - \vec{S}(\vec{B})\vert  &\le C_1 \,
		\phi'_{\abs{\vec{A}^{\textup{sym}}}}(\abs{\vec{A}^{\textup{sym}} -
			\vec{B}^{\textup{sym}}})\,.\label{assum:extra_stress.2}
	\end{align}
	The~constants $C_0,C_1>0$ and $p\in (1,\infty)$ are called the \textup{characteristics} of $\vec{S}$.
\end{assum}


Closely related to the extra stress tensor $\vec{S}\colon \R^{d \times d} \to \R^{d \times d}_{\textup{sym}} $ with $(p,\delta)$-structure is the non-linear function $\vec{F} 
\colon\R^{d\times d}\to \R^{d\times d}_{\textup{sym}}$, for every $\vec{A}\in \mathbb{R}^{d\times d}$  defined by 
\begin{align}
	\begin{aligned}
		\vec{F}(\vec{A})&\coloneqq (\delta+\vert \vec{A}^{\textup{sym}}\vert)^{\smash{\frac{p-2}{2}}}\vec{A}^{\textup{sym}}\,. 
	\end{aligned}
	\label{eq:def_F}
\end{align}
The connections between
$\vec{S},\vec{F}\colon \R^{d \times d} \to \R^{d\times d}_{\textup{sym}}$ and
$\phi_a,(\phi_a)^*\colon \R^{\ge 0}\hspace{-0.05em}\to\hspace{-0.05em} \R^{\ge 0}$,~${a\ge  0}$, are best presented
by the following result (cf.~\cite{DE08_fractional,GNSE_DR07,DKRT13_ldg}).

\begin{prop}
	\label{lem:growth_SF}
	Let $\vec{S}$ satisfy Assumption~\ref{assum:extra_stress}, let $\phi$ be defined in \eqref{eq:fem:nfunction}, and let $\vec{F}$ be defined in~\eqref{eq:def_F}. Then, uniformly with respect to 
	$\vec{A}, \vec{B} \in \R^{d \times d}$, we have that
	\begin{align}\label{eq:growth_S}
		\begin{aligned}
			(\vec{S}(\vec{A}) - \vec{S}(\vec{B}))
			:(\vec{A}-\vec{B} ) &\sim  \abs{ \vec{F}(\vec{A}) - \vec{F}(\vec{B})}^2
			\\
			&\sim \phi_{\vert \vec{A}^{\textup{sym}}\vert }(\vert \vec{A}^{\textup{sym}}
			- \vec{B}^{\textup{sym}}\vert )
			\\
			&\sim(\phi_{\vert\vec{A}^{\textup{sym}} \vert})^*(\vert\vec{S}(\vec{A} ) - \vec{S}(\vec{B} )\vert)\,.
		\end{aligned}
	\end{align}
	The constants in \eqref{eq:growth_S} depend only on the characteristics of ${\vec{S}}$.
\end{prop}

The following results can be found in~\cite{DK08_adaptive,GNSE_DR07}.

\begin{lemma}[change of shift]\label{lem:shift_change}
	Let $\phi$ be defined in \eqref{eq:fem:nfunction} and let $\vec{F}$ be defined in \eqref{eq:def_F}. Then,
	for every $\epsilon>0$, there exists $c_\epsilon\geq 1$ (depending only
	on~$\epsilon>0$ and the characteristics of $\phi$) such that for every $\vec{A},\vec{B}\in\smash{\R^{d \times d}_{\textup{sym}}}$ and $t\geq 0$, it holds that
	\begin{align}\label{lem:shift_change.1}
		\smash{\phi_{\vert \vec{B}\vert}(t)}&\leq \smash{c_\epsilon\, \phi_{\vert\vec{A}\vert}(t)
			+\epsilon\, \vert \vec{F}(\vec{B}) - \vec{F}(\vec{A})\vert ^2\,,}
		\\
		\smash{(\phi_{\vert\vec{B}\vert})^*(t)}&\leq \smash{c_\epsilon\, (\phi_{\vert\vec{A}\vert})^*(t)
			+\epsilon\,\vert\vec{F}(\vec{B}) - \vec{F}(\vec{A})\vert^2}\,.\label{lem:shift_change.2}
	\end{align}
\end{lemma}

\subsection{Continuous weak formulation}

In this subsection, we give a precise (weak) problem formulation of the  (strong) problem formulation \eqref{eq:fem:main_problem1}--\eqref{eq:fem:main_problem3}.
If $p < \tfrac{3d}{d+2}$, the theory of pseudo-monotone operators (cf. \cite{Zeidler2B}) is no longer applicable, since  we, then,
have to consider test functions with higher regularity than the standard~weak~re\-gu\-larity. For every $p\in (1,\infty)$, we define the coefficient $s \coloneqq  s(p) \coloneqq  \max \{ p, (\tfrac{p\d}{2})' \}$. Then, a first (weak) problem formulation of the  (strong) problem formulation \eqref{eq:fem:main_problem1}--\eqref{eq:fem:main_problem3}~is~given~via:

\textsc{Problem (\hypertarget{Q}{Q}).} Given $\vec{f}\in V\d$, $ g_1  \in L^s(\Omega)$, and $ \vec{g}_2\in W^{\smash{1-\frac{1}{s}},s}(\partial\Omega)$, 
find $(\vec{v}, q)^\top \in X \times Q^s$ such that 
\begin{align}
	(\vec{S}(\D{v}), \D{z})_\Omega - (\vec{v} \otimes \vec{v}, \nabla \vec{z})_\Omega  - (q, \div \vec{z})_\Omega  &= (\vec{f}, \vec{z})_\Omega   &&\text{ for all } \vec{z} \in V^s\,, \label{eq:fem:q_main} \\
	\div \vec{v} &= g_1  &&\text{ in } L^s(\Omega) \,, \label{eq:fem:q_div} \\
	\tr \vec{v} &= \vec{g}_2  &&\text{ in } W^{\smash{1-\frac{1}{s}},s}(\del\Omega)\,. \label{eq:fem:q_bdry}
\end{align}

For the moment, let us assume that there exists a vector field $\vec{g} \in X^s$ (an explicit construction is given in  Lemma \ref{lem:fem:gh}), which solves the system
\begin{align}\begin{aligned} \label{eq:fem:div_eq}
		\div \vec{g} &= g_1&&\quad\text{ in }L^s(\Omega)\,, \\
		\tr \vec{g} &= \vec{g}_2&&\quad\text{ in }W^{\smash{1-\frac{1}{s}},s}(\partial\Omega)\,.
\end{aligned}\end{align}
The well-posedness of the system \eqref{eq:fem:div_eq} will be established  in Lemma~\ref{lem:fem:gh}.~Then,~setting $$\vec{u} \coloneqq  \vec{v} - \vec{g} \in V_0\,,$$
where $V_0 \coloneqq  V_0^p$ and $V_0^r \coloneqq  \{ \vec{z} \in V^r \mid\div \vec{z} = 0 \text{ a.e.\ in }\Omega\}$ for all $r\in [1,\infty)$,
an inf-sup argument yields the following equivalent formulation (cf.\ \cite[Cor.~A.3]{Tscherpel_phd}):

\textsc{Problem (\hypertarget{P}{P}). }
Given $\vec{f}\in V\d$, $ g_1  \in L^s(\Omega)$, and $ \vec{g}_2\in W^{\smash{1-\frac{1}{s}},s}(\partial\Omega)$,~find~$\vec{u} \in V_0$ such that
\begin{align*}
	(\vec{S}(\D{u}+\D{g}), \D{z})_\Omega - ((\vec{u}+\vec{g}) \otimes (\vec{u}+\vec{g}), \nabla \vec{z})_\Omega  = (\vec{f}, \vec{z})_\Omega\quad \text{ for all }\vec{z} \in V_0^s\,.
\end{align*}

\subsection{Triangulations and finite element spaces}

In this subsection, we specify our assumptions on the triangulations and the finite element spaces employed in the discretization of Problem (\hyperlink{Q}{Q}) (or equivalently of Problem (\hyperlink{P}{P})).

\begin{assum}[triangulation] \label{assum:fem:geometry} 
	We assume that
	$\{\mathcal{T}_h\}_{h>0}$ is a family of conforming
	triangulations of $\overline{\Omega}\subseteq \mathbb{R}^d$,
	$d\in \{2,3\}$, cf.\ \cite{BS08},  consisting of
	\mbox{$d$-dimensional} simplices $K$.
	The parameter $h>0$ refers to the \textup{maximal mesh-size} of $\mathcal{T}_h$, i.e., if we define $h_K\coloneqq \textup{diam}(K)$ for all $K\in \mathcal{T}_h$, then ${h\coloneqq \max_{K\in \mathcal{T}_h}{h_K}}$.
	For a simplex $K \in \mathcal{T}_h$,
	we denote the \textup{supremum of diameters of inscribed balls} by $\rho_K>0$. We assume that there is a constant $\omega_0>0$,~\mbox{independent}~of~$h>0$, such that ${h_K}{\rho_K^{-1}}\le
	\omega_0$~for~all~${K \in \mathcal{T}_h}$. The smallest such constant is called the \textup{chunkiness} of $\{\mathcal{T}_h\}_{h>0}$. 
\end{assum}

For every $K\hspace*{-0.1em} \in\hspace*{-0.1em} \mathcal{T}_h$, the \textit{element patch} is defined by ${\omega_K\hspace*{-0.1em}\coloneqq\hspace*{-0.1em} \{K'\in \mathcal{T}_h\mid \partial K'\cap \partial K\hspace*{-0.1em}\neq \hspace*{-0.1em}\emptyset\}}$.

\begin{defi}[finite element spaces] \label{defi:fem:spaces}
	The space of polynomials of degree at most $r\in \mathbb{N}\cup\{0\}$ on each simplex is denoted as $\mathbb{P}^r(\mathcal{T}_h)$. We assume that the finite element spaces $X_h \subseteq \mathbb{P}^m(\mathcal{T}_h)$ and $Y_h \subseteq \mathbb{P}^k(\mathcal{T}_h)$ are conforming, i.e.,~${X_h \subseteq X}$~and~${Y_h \subseteq Y}$.
	We define $V_h \coloneqq  X_h \cap V$ and $Q_h \coloneqq  Y_h \cap Q$.
\end{defi}

The following assumption becomes relevant in the construction of a discrete (divergence-corrected) Lipschitz truncation  (cf.~\cite{DKS13a} or \cite[Ass. 2.20]{Tscherpel_phd}).
\begin{assum}[locally supported basis of $Y_h$] \label{assum:fem:localbasis}
	We assume that $Y_h$ has a locally supported basis $Y_h = \spann \{q_h^1, \dots, q_h^l \}$ such that $q^i_h|_K \neq 0$ implies  that $\supp q_h^i \subseteq \omega_K$ for all $i \in \{1, \dots, l\}$ and $K \in \mathcal{T}_h$.
\end{assum}

\begin{assum}[projection for the discrete pressure space] \label{assum:fem:projection_y}
	We suppose that  $\mathbb{R}\subseteq Y_h$ 
	and that there exists a linear projection operator $\Pi_h^Y \colon Y \to Y_h$ such that for every $z \in Y$ and $K \in \mathcal{T}_h$, it holds that
	\begin{align*}
		\langle \vert \Pi_h^Y\! z\vert \rangle_K\leq c\, \langle \vert z\vert \rangle_{\omega_K}\,.
	\end{align*}
\end{assum}

The local stability of the projection operator $\Pi_h^Y$ implies the convergence of the images as $h\to 0$.
\begin{lemma}[convergence of pressure projection] \label{lem:fem:appr_y}
	Assumption~\ref{assum:fem:projection_y} implies that $\Pi_h^Y\! z \to z$ in $ L^r(\Omega)$ $(h \to 0)$ for all $z \in L^r(\Omega)$ and $r\in [1,\infty)$.
\end{lemma}
\begin{proof}
	Let $\varepsilon>0$ be fixed, but arbitrary. Moreover, let
	$z_\varepsilon \in W^{1,p}(\Omega)$ be such that $\| z- z_\varepsilon\|_{p} \leq \varepsilon$, which is possible due to the density of smooth functions in $L^p(\Omega)$. Thus, due to \cite[Thms. 4.5, 4.6]{DR07_interpolation}, for every $\varepsilon>0$ and $h\in (0,1]$, we have  that
	\begin{align*}
		\| z - \Pi_h^Y\! z\|_r
		&\leq \| z - z_\varepsilon\|_r + \| z_\varepsilon- \Pi_h^Y\! z_\varepsilon\|_r + \|\Pi_h^Y (z_\varepsilon - z)\|_r
		\\&\leq  (1+c)\,\varepsilon+c \,h \,\|z_\varepsilon\|_{1,r}\,.
	\end{align*}
	In other words, for arbitrary $\varepsilon>0$, it holds that $\limsup_{h \to 0}{\| z - \Pi_h^Y\! z\|_p}\leq \varepsilon$, i.e., $\Pi_h^Y\! z \to z$ in $L^r(\Omega)$ $(h \to 0)$.
\end{proof}

\begin{lemma} \label{lem:fem:comparable_qh}
	Let $z \in Q$ and $\psi$ be an N-function. Then, the approximation by functions in $Q_h$ and $Y_h$ is comparable, i.e., for every $z_h \in Y_h$, we have that
	\begin{align*}
		\rho_{\psi,\Omega }(z - (z_h - \mean{z_h}_{\Omega})) \leq 2\, \rho_{\psi,\Omega }(z - z_h)\,.
	\end{align*}
\end{lemma}
\begin{proof}
	Due to the convexity of $\psi$, the $\Delta_2$-condition satisfied by $\psi$,  $\mean{z}_{\Omega} = 0$, and Jensen's inequality, we have that
	\begin{align*}
		\rho_{\psi,\Omega }(z - (z_h - \mean{z_h}_{\Omega}))
		&\leq \rho_{\psi,\Omega }(z - z_h) + \rho_{\psi,\Omega }(\mean{z - z_h}_{\Omega})\\&
		\leq 2\,\rho_{\psi,\Omega }(z - z_h) \,.
	\end{align*}
\end{proof}

\begin{assum}[projections for the discrete velocity space] \label{assum:fem:projection_x}
	We suppose that $\mathbb{P}^1(\mathcal{T}_h) \subset X_h$ and that there exists
	a linear projection operator $\Pi_h^X \colon X \to X_h$ with the following properties:\enlargethispage{6mm}
	\begin{itemize}
		\item[(i)] \textup{Local $W^{1,1}$-stability:} for every $\vec{z} \in X$ and $K \in \mathcal{T}_h$, it holds that
		\begin{align*}
			\langle \vert\Pi_h^X \vec{z}\vert \rangle_K  \leq c\,\langle \vert\vec{z}\vert \rangle_{\omega_K} + c\, \langle h_K \vert \nabla \vec{z}\vert  \rangle_{\omega_K} \,.
		\end{align*}
		\item[(ii)] \textup{Preservation of zero boundary values:} It holds that $\Pi_h^X(V) \subseteq V_h$,
		\item[(iii)] \textup{Preservation of divergence in the $Y_h^*$-sense:} For every $\vec{z} \in X$ and  $z_h \in Y_h$, it holds that\vspace*{-2.5mm}
		\begin{align*}
			(\div \vec{z}, z_h)_\Omega = (\div \Pi_h^X \vec{z}, z_h)_\Omega \,.
		\end{align*}
	\end{itemize}
\end{assum}

Analogously to Lemma~\ref{lem:fem:appr_y}, we have that
\begin{cor}[convergence of velocity projection] \label{cor:fem:appr_x}
	Assumption~\ref{assum:fem:projection_x} implies that $\Pi_h^X \vec{z} \to \vec{z}$ 
	in $W^{1,r}(\Omega)$ $(h \to 0)$ for all $\vec{z} \in W^{1,r}(\Omega)$ and $r\in [1,\infty)$.
\end{cor}

Next, we present a short list of common mixed finite element spaces $\{X_h\}_{h>0}$ and $\{Y_h\}_{h>0}$ with projectors $\{\Pi_h^X\}_{h>0}$ and $\{\Pi_h^Y\}_{h>0}$ on regular triangulations $\{\mathcal{T}_h\}_{h>0}$ satisfying both  Assumption~\ref{assum:fem:projection_y}~and~Assumption~\ref{assum:fem:projection_x}; see also \cite{Tscherpel_phd} or \cite{BBF13}.

\begin{remark}\label{FEM.Q}
	{\itshape
		The following discrete spaces and projectors satisfy Assumption~\ref{assum:fem:projection_y}:
		\begin{description}[noitemsep,topsep=2pt,leftmargin=!,labelwidth=\widthof{(iii)},font=\normalfont\itshape]
			\item[(i)] If $Y_h= \mathbb{P}^l(\mathcal{T}_h)$ for some $l\ge 0$, then $\Pi_h^Y$ can be chosen as (local) $L^2$-projection operator or, more generally, a Cl\'ement type quasi-interpolation operator.
			
			\item[(ii)] If  $Y_h=\mathbb{P}^l(\mathcal{T}_h)$  for  some  $l\ge 1$,  then  $\Pi_h^Y$  can  be  chosen  as  a  Cl\'ement  type  quasi-interpolation operator.
			
	\end{description}}
\end{remark}

\begin{remark}\label{FEM.V}
	{\itshape
		The following discrete spaces  and projectors satisfy Assumption~\ref{assum:fem:projection_x}:
		\begin{description}[noitemsep,topsep=2pt,leftmargin=!,labelwidth=\widthof{(iii)},font=\normalfont\itshape]
			\item[(i)] The \textup{MINI element} for $d\in \{2,3\}$, i.e., $X_h=\mathbb{P}^1_c(\mathcal{T}_h)^d\bigoplus\mathbb{B}(\mathcal{T}_h)^d$, where $\mathbb{B}(\mathcal{T}_h)$ is the bubble function space, and $Y_h=\mathbb{P}^1_c(\mathcal{T}_h)$, introduced in \cite{ABF84} for $d=2$; see also \cite[Chap.\ II.4.1]{GR86} and \cite[Sec.\ 8.4.2, 8.7.1]{BBF13}. A proof of Assumption \ref{assum:fem:projection_x}
			is given in \cite[Appx.\ A.1]{BBDR12_fem}; see also \cite[Lem.\ 4.5]{GL01}~or~\mbox{\cite[p.~990]{DKS13}}.
			
			\item[(ii)] The \textup{Taylor--Hood element} for $d\in\{2,3\}$, i.e., $X_h=\mathbb{P}^2_c(\mathcal{T}_h)^d$ and $Y_h=\mathbb{P}^1_c(\mathcal{T}_h)$, introduced in \cite{TH73} for $d=2$; see  also \cite[Chap.\ II.4.2]{GR86}, and its generalizations; see, e.g., \cite[Sec.\ 8.8.2]{BBF13}. A proof
			of Assumption \ref{assum:fem:projection_x} is given in \cite[Thm.~3.1,~3.2]{GS03}.
			
			\item[(iii)] The \textup{conforming Crouzeix--Raviart element} for $d\hspace*{-0.1em}=\hspace*{-0.1em}2$, i.e., $X_h=\mathbb{P}^2_c(\mathcal{T}_h)^2\bigoplus\mathbb{B}(\mathcal{T}_h)^2$ and ${Y_h=\mathbb{P}^1(\mathcal{T}_h)}$, introduced in \cite{CR73}; see also \cite[Ex.\ 8.6.1]{BBF13}. The operator $\Pi_h^X$ satisfying Assumption~\ref{assum:fem:projection_x}~(i) is given in \cite[p.\ 49]{CR73} and it can be shown to satisfy Assumption \ref{assum:fem:projection_x} (ii); see, e.g., \cite[Thm.\ 3.3]{GS03}.
			
			\item[(iv)] The \hspace*{-0.1mm}\textup{first \hspace*{-0.1mm}order \hspace*{-0.1mm}Bernardi--Raugel \hspace*{-0.1mm}element} \hspace*{-0.1mm}for \hspace*{-0.1mm}$d\hspace*{-0.175em} \in \hspace*{-0.175em}  \{2,3\}$, \hspace*{-0.1mm}i.e.,\hspace*{-0.15em} 
			$V_h \hspace*{-0.175em}= \hspace*{-0.175em}\mathbb{P}^1_c(\mathcal{T}_h)^d\hspace*{-0.175em}\bigoplus\hspace*{-0.175em}\mathbb{B}_{\tiny \mathscr{F}}(\mathcal{T}_h)^d\hspace*{-0.175em}$, where $\mathbb{B}_{\tiny \mathscr{F}}(\mathcal{T}_h)$ is the facet bubble function space, and $Y_h=\mathbb{P}^0(\mathcal{T}_h)$, introduced in \cite[Sec. II]{BR85}. For $d=2$ is often  referred to as \textup{reduced $\mathbb{P}^2$-$\mathbb{P}^0$-element} or as \textup{2D SMALL element}; see,  e.g.,  \cite[Rem.\ 8.4.2]{BBF13} and \cite[Chap.\ II.2.1]{GR86}.  The operator $\Pi_h^X$ satisfying Assumption \ref{assum:fem:projection_x} is given in
			\cite[Sec.\ II.]{BR85}.
			
			\item[(v)] The  {\textup{second  order  Bernardi--Raugel  element}}  for  $d=3$,  introduced  in  \cite[Sec.~III]{BR85};  see  also \mbox{\cite[Ex.\ 8.7.2]{BBF13}}  and  \cite[Chap.\  II.2.3]{GR86}.  The  operator  $\Pi_h^X$  satisfying  Assumption  \ref{assum:fem:projection_x}  is  given in \mbox{\cite[Sec. III.3]{BR85}}; see also \cite{Tscherpel_phd}.
	\end{description}}
\end{remark}

The following discrete inf-sup stability is crucial for establishing~a~discrete~pressure:
\begin{lemma}[discrete inf-sup condition] \label{lem:fem:discrete-inf-sup}
	Assumption~\ref{assum:fem:projection_x} implies that for every $q_h \in Q_h$, 
	it holds that
	\begin{align*}
		\|q_h\|_{p',\Omega} \leq c \sup_{\vec{z}_h \in X_h\colon \|\vec{z}_h\|_{1,p,\Omega} \leq 1}{ (q_h,\div \vec{z}_h)_\Omega}\,,
	\end{align*}
	where the constant $c>0$ depends on only on $m$, $k$, $p$, and $\omega_0$, and $\Omega$. 
\end{lemma}

\begin{proof}
	See \cite[Lems. 4.1, 4.3]{BBDR12_fem}. 
\end{proof}

\subsection{Discrete weak formulation} \label{sub:fem:discretization}

Classically, skew-symmetry plays a crucial role in the analysis of the convective term. This property seems out of reach in the case of a non-zero divergence constraint on the velocity. Nevertheless, we may preserve certain cancellation properties. Therefore, we employ discretizations ${b,\widetilde{b}\colon (X)^2\times X^s \to \mathbb{R}}$, commonly referred to as \textit{Temam's modifications},~for~every~${(\vec{u},\vec{v},\vec{w})^\top\in (X)^2\times X^s}$, (cf. \cite{Temam84, KR23_ldg1}) defined by 
\begin{align}\label{def:b}
	\begin{aligned}
		b(\vec{u},\vec{v},\vec{w}) 
		&\coloneqq  \tilde{b}(\vec{u},\vec{v},\vec{w}) + \tfrac{1}{2} (g_1 \vec{u}, \vec{w})_\Omega \\
		&\coloneqq  \tfrac{1}{2} (\vec{w}\otimes\vec{u}, \nabla \vec{v} + g_1 \vec{I})_\Omega - \tfrac{1}{2} (\vec{v}\otimes\vec{u}, \nabla \vec{w})_\Omega\,.
	\end{aligned}
\end{align}
\begin{remark}
	Traditionally, i.e., in the context of a prescribed homogeneous divergence, the trilinear form $\widetilde{b}\colon (X)^2\times X^s \to \mathbb{R}$ is referred to as Temam's modification. In~\cite{KR23_ldg1}, this is extended to the case of inhomogeneous divergence: subtracting the term $\tfrac{1}{2} (g_1 \vec{u}, \vec{w})_\Omega$ yields consistency with the convective term in the convective formulation. Since the conservative formulation as in~\eqref{eq:fem:main_problem1} seems to be more physically reasonable, we chose the above version (with a different sign) in order to be consistent with~that~version. The numerical analysis as well as the corresponding experimental results are comparable for both variants. 
\end{remark}

\begin{lemma} \label{lem:fem:admissible}
	Let $p \geq \smash{\tfrac{2d}{d+1}}$. Then, Temam's modifications ${b,\widetilde{b}\colon (X)^2\times X^s \to \mathbb{R}}$ are well-defined and bounded. Moreover, $b\colon (X)^2\times X^s \to \mathbb{R}$
	is consistent, i.e., for every $\vec{v} \hspace*{-0.1em}\in\hspace*{-0.1em} W^{1,p}(\Omega)$ with $\div \vec{v} \hspace*{-0.1em}=\hspace*{-0.1em} g_1$ and ${\vec{z}\hspace*{-0.1em} \in \hspace*{-0.1em}W_0^{1,s}(\Omega)}$,~it~holds~that
	\begin{align} \label{eq:fem:consistency}
		b(\vec{v}, \vec{v}, \vec{z}) = - (\vec{v}\otimes\vec{v}, \nabla\vec{z})_\Omega\,.
	\end{align}
	Moreover, for every $\vec{v} \in W^{1,p}(\Omega)$ and $\vec{z} \in W^{1,s}(\Omega)$, it holds that
	\begin{align} \label{eq:fem:skew}
		b(\vec{v}, \vec{z}, \vec{z}) = \tfrac{1}{2} (g_1 \vec{v}, \vec{z})_\Omega\,.
	\end{align}
\end{lemma}
\begin{proof}
	Since for $p \geq \smash{\tfrac{2d}{d+1}}$, we have that $\tfrac{1}{p} + \tfrac{1}{p\d} + \frac{1}{s\d} \leq 1$, 
	well-definedness and
	boundedness follow from Hölder's inequality and Sobolev's embedding theorem. Then,
	consistency follows from integration-by-parts and the vanishing traces of~test~functions. The last assertion \eqref{eq:fem:skew} is a direct consequence of the definition \eqref{def:b}.
\end{proof}

Given \eqref{def:b}, we consider the following discrete counterpart of Problem (\hyperlink{Q}{Q}):

\textsc{Problem} (\hypertarget{Qh}{Q$_h$}). 
Given $\vec{f}\hspace*{-0.15em}\in\hspace*{-0.15em}V\d$, $g_1\hspace*{-0.15em}\in\hspace*{-0.15em} L^{s'}(\Omega)$, $\vec{g}_2\hspace*{-0.15em}\in\hspace*{-0.15em} W^{1-\frac{1}{s},s}(\partial\Omega)$,~and~${\vec{g}_b \hspace*{-0.15em}\in\hspace*{-0.15em} W^{1,s}(\Omega)}$ with $\vec{g}_b=\vec{g}_2$ a.e.\ in $\partial\Omega$, 
find $(\vec{v}_h, q_h)^\top \in X_h \times Q_h$ such that
\begin{align}
	(\vec{S}(\D{v}_h), \D{z}_h)_\Omega + b(\vec{v}_h, \vec{v}_h,\vec{z}_h) - (q_h, \div \vec{z}_h)_\Omega  &= (\vec{f}, \vec{z}_h)_\Omega   &&\textrm{ for all } \vec{z}_h \in V_h\,, \hspace*{-2.5mm}\label{eq:fem:qh_main} \\
	(\div \vec{v}_h, z_h) _\Omega &= (g_1, z_h)_\Omega   && \textrm{ for all } z_h \in Y_h\,,\hspace*{-2.5mm} \label{eq:fem:qh_div} \\
	\vec{v}_h &= \smash{\vec{g}_2^h} && \textrm{ in } \tr X_h\,, \label{eq:fem:qh_bdry}
\end{align}
where $\smash{\vec{g}_2^h}\coloneqq  \Pi_h^X \vec{g}_b \in X_h$.

\begin{remark}
	The properties of the Fortin interpolation operator $\Pi_h^X\colon X\to X_h$ (cf. Assumption \ref{assum:fem:projection_x}) yield that the Dirichlet boundary condition \eqref{eq:fem:qh_bdry} does not depend on the choice of the trace lift $\vec{g}_b \in W^{1,s}(\Omega)$. The fact that $\Pi_h^X\colon X\to X_h$ (cf. Assumption \ref{assum:fem:projection_x}(iii)) preserves the divergence in the $Y_h^*$-sense implies the discrete compatibility of $g_1\in L^p(\Omega)$ and $\vec{g}_2^h\in X_h$, i.e., for every $z_h\in Y_h$, it holds that
	\begin{align*}
		(\div \vec{g}_2^h, z_h) _\Omega = (g_1, z_h)_\Omega   \,.
	\end{align*} 
\end{remark}

In order to establish the existence of discrete solutions to Problem (\hyperlink{Qh}{Q$_h$}),~an~equivalent~hydromechanical~formulation --\textit{``hiding''} the pressure-- will prove to be  advantageous. As a preparation, we construct a \textit{discrete lift} $\vec{g}_h \in X_h$ of the discrete Dirichlet boundary condition $\smash{\vec{g}_2^h}\in X_h$, i.e., the discrete divergence constraint \eqref{eq:fem:qh_div}  and the discrete Dirichlet boundary condition \eqref{eq:fem:qh_bdry}.
\begin{lemma} \label{lem:fem:gh}
	There exists $\vec{g}_h \in X_h$ such that
	\begin{align}\label{lem:fem:gh.1}
		\begin{aligned}
			(\div \vec{g}_h, z_h)_\Omega &= (g_1, z_h)_\Omega \quad &&\textrm{ for all } z_h \in Y_h\,, \\
			\vec{g}_h &= \smash{\vec{g}_2^h}&&\textrm{ in } \tr X_h\,.
		\end{aligned}
	\end{align}
	In addition, for  $\vec{g}\coloneqq  \vec{g}_b + \Bog_\Omega (g_1 - \div \vec{g}_b)\in W^{1,s}(\Omega)$, where $\Bog_\Omega\colon L^s_0(\Omega)\to W^{1,s}_0(\Omega)$ denotes the \textit{Bogovski\u{\i} operator}, it holds \eqref{eq:fem:div_eq} and that
	\begin{align}\label{lem:fem:gh.2}
		\vec{g}_h \to \vec{g} \quad \text{ in }W^{1,s}(\Omega)\quad (h\to 0)\,.
	\end{align}
\end{lemma}
\begin{proof}
	Inasmuch as
	$g_1 - \div \smash{\vec{g}_2^h}\in L^s_0(\Omega)$, by the properties of the Bogovski\u{\i}~operator, the vector field
	\begin{align*}
		\vec{g}_h \coloneqq  \smash{\vec{g}_2^h}+ \Pi_h^X \Bog_\Omega ({g_1} - \div \vec{g}_2^h)\quad \text{ in } W^{1,s}(\Omega)\,,
	\end{align*}
	satisfies \eqref{lem:fem:gh.1}. \!In addition,  by the properties of the Bogovski\u{\i} operator and  Lemma~\ref{cor:fem:appr_x}, for 
	$\vec{g}\coloneqq  \vec{g}_b + \Bog_\Omega (g_1 - \div \vec{g}_b)\in W^{1,s}(\Omega)$, we conclude that  \eqref{eq:fem:div_eq}  and \eqref{lem:fem:gh.2} apply. 
\end{proof}

In order to establish the existence of a discrete velocity vector field in Problem~(\hyperlink{Qh}{Q$_h$}), we prefer a problem formulation \textit{``hiding''} the discrete pressure. More precisely, taking the discrete lift $\vec{g}_h\in X_h$ from Lemma \ref{lem:fem:gh} and making~the~ansatz
\begin{align}\label{eq:discrete_ansatz}
	\vec{u}_h \coloneqq  \vec{v}_h - \vec{g}_h \in V_{h,0}\,,
\end{align}
leads to the following discrete counterpart of Problem (\hyperlink{P}{P}):\enlargethispage{5mm}

\textsc{Problem} (\hypertarget{Ph}{P$_h$}). 
Given $\vec{f}\hspace*{-0.15em}\in\hspace*{-0.15em}V\d$, $g_1\hspace*{-0.15em}\in\hspace*{-0.15em} L^{s'}(\Omega)$, $\vec{g}_2\hspace*{-0.15em}\in\hspace*{-0.15em} W^{1-\frac{1}{s},s}(\partial\Omega)$,~and~${\vec{g}_b \hspace*{-0.15em}\in\hspace*{-0.15em} W^{1,s}(\Omega)}$ with $\vec{g}_b=\vec{g}_2$ a.e.\ in $\partial\Omega$, 
find $\vec{u}_h \in V_{h,0}$ such that
\begin{align*}
	(\vec{S}(\D{u}_h+\D{g}_h), \D{z}_h)_\Omega + b(\vec{u}_h+\vec{g}_h, \vec{u}_h+\vec{g}_h,\vec{z}_h)  = (\vec{f}, \vec{z}_h)_\Omega \quad\text{ for all }\vec{z}_h \in V_{h,0}\,.
\end{align*}

The equivalence of Problem (\hyperlink{Qh}{Q$_h$}) and Problem (\hyperlink{Ph}{P$_h$}) follows  from Lemma~\ref{lem:fem:gh} and the discrete inf-sup condition (cf. Lemma \ref{lem:fem:discrete-inf-sup}).

%% file: section2.tex
\section{Existence and convergence of discrete solutions} \label{sec:fem:ex_conv}

In this section, we establish the well-posedness of Problem (\hyperlink{Qh}{Q$_h$}) (or equivalently of Problem (\hyperlink{Ph}{P$_h$})) and the (weak) convergence of a discrete solution $(\vec{v}_h,q_h)^\top\in X_h\times Q_h$ of Problem (\hyperlink{Qh}{Q$_h$}) to a solution $(\vec{v},q)^\top\in X\times Q^s$ of Problem (\hyperlink{Q}{Q}). To this end, we need to restrict either to the case that $p>2$ or to the case of sufficiently \textit{``regular and small''} data.

\begin{assum} \label{assum:fem:smallness}
	We assume that $p>2$ or that the data are sufficiently regular and small, i.e., there exists a sufficiently small constant $c_{\textup{data}}>0$, depending only on the characteristics of $\vec{S}$, $\vec{f}$, and $\Omega$, such that $\norm{g_1}_s + \norm{\vec{g_2}}_{\smash{1-\frac{1}{s},s}} \leq c_{\textup{data}}$.
\end{assum}
\begin{remark}
	The additional regularity implicated by Assumption~\ref{assum:fem:regularity} becomes relevant in the case $p < \tfrac{3d}{d+2}$ and is discussed in~\cite{JR21_inhom}.
\end{remark}

\begin{prop} \label{prop:fem:existence}
	Let Assumptions \ref{assum:fem:geometry}, \ref{assum:fem:projection_x}, \ref{assum:extra_stress} and \ref{assum:fem:smallness} be satisfied and let $p \geq \tfrac{2d}{d+1}$.
	Then, there exists a solution $(\vec{v}_h, q_h)^\top \in X_h\times Q_h$ to Problem (\hyperlink{Qh}{Q$_h$}) and
	\begin{align} \label{eq:fem:apriori}
		\|\vec{v}_h\|_{1,p} + \|q_h\|_{s'} \leq c\,,
	\end{align}
	where $c>0$ is a constant depending only on the characteristics of $\vec{S}$, $\omega_0$, $m$, $k$, $\|\vec{f}\|_{V^*}$, $\|g_1\|_{s,\Omega}$, $\|\vec{g}_2\|_{\smash{1-\frac{1}{s},s,\Omega}}$, and $\Omega$.
\end{prop}
\begin{proof}
	Suppose $\|g_1\|_s + \|\vec{g_2}\|_{\smash{1-\frac{1}{s},s}} \leq R$ for some $R >0$. According to Lemma~\ref{lem:fem:gh}, we, then, have that $\|\vec{g}_h\|_{1,s,\Omega} \hspace*{-0.1em}\leq\hspace*{-0.1em} c \, R$, where $c\hspace*{-0.1em}>\hspace*{-0.1em}0$ is a constant depending~only~on~$\omega_0$,~$m$,~$s$, and $\Omega$.
	We intend to apply the \textit{main theorem on locally coercive operators} (cf.\ \cite[Thm. 27.B]{Zeidler2B}) to the non-linear operator $A_h\colon \hspace*{-0.1em}V_{h,0}\hspace*{-0.1em}\to\hspace*{-0.1em} (V_{h,0})^*$, for every $\vec{z}_h,\vec{w}_h\hspace*{-0.1em}\in\hspace*{-0.1em} V_{h,0}$,~defined~by
	\begin{align*}
		\langle A_h\vec{z}_h,\vec{w}_h\rangle_{V_{h,0}}\coloneqq (\vec{S}(\D{z}_h+\D{g}_h), \D{w}_h)_\Omega+b(\vec{z}_h+\vec{g}_h, \vec{z}_h+\vec{g}_h,\vec{w}_h)-(\vec{f}, \vec{w}_h)_\Omega\,,
	\end{align*}
	to establish that $0\in \textup{im}(A_h)$, i.e.,  the existence of a solution $\vec{u}_h\in V_{h,0}$ to Problem (\hyperlink{Ph}{P$_h$}), if $p>2$ or if $R>0$ is sufficiently small. Since, due to the finite-dimensionality of $V_{h,0}$, it is readily seen that $A_h\colon V_{h,0}\to (V_{h,0})^*$ is well-defined and continuous, it is only left to show that $A_h\colon V_{h,0}\to (V_{h,0})^*$ is \textit{locally coercive}, i.e., that there exists a constant ${r\hspace*{-0.1em}>\hspace*{-0.1em} 0}$ such that for every $\vec{z}_h\in V_{h,0}$ from $\|\D{z}_h\|_{p,\Omega}=r$, it~follows~that~${\langle A_h\vec{z}_h,\vec{z}_h\rangle_{V_{h,0}}\ge 0}$.

	Due to~\cite[Lem.\ 2.13]{JR21_inhom}, for every $\vec{z}_h\in V_{h,0}$, we have hat
	\begin{align} \label{prop:fem:existence.1}
		(\vec{S}(\D{z}_h+\D{g}_h), \D{z}_h)_\Omega
		\geq c_1 \,\|\D{z}_h\|_{p,\Omega}^p - c_2 \,\|\D{z}_h\|_{p,\Omega}\,,
	\end{align}
	where $c_1  >0$ depends only on the characteristics of $\vec{S}$ and  $c_2 >0$, in addition,~on~$s$, $\|g_1\|_{s,\Omega}$, $\|\vec{g}_b\|_{1,s,\Omega}$, $\Omega$, $m$, and $\omega_0$.
	With the skew-symmetry property of Temam's modification, Hölder's,~Sobolev's, and Korn's inequality, for every $\vec{z}_h\in V_{h,0}$,~we~find~that
	\begin{align} \label{prop:fem:existence.2}
		\begin{aligned}
			\vert b(\vec{z}_h+\vec{g}_h, \vec{z}_h+\vec{g}_h,\vec{z}_h)\vert 
			&\leq c_3\,\|\vec{g}_h\|_{1,s,\Omega} \,\|\D{z}_h\|_{p,\Omega}^2 + c_4\, \|\D{z}_h\|_{p,\Omega}\,,\\
			\vert (\vec{f}, \vec{z}_h)_\Omega\vert &  \leq c_5\, \|\vec{f}\|_{V\d}\, \|\D{z}_h\|_{p,\Omega}\,,
		\end{aligned}
	\end{align}
	where $c_3, c_5>0$ depend only on $p$ and $\Omega$ and $c_4>0$, in addition, on $s$, $\|g_1\|_{s,\Omega}$, $\|\vec{g}_b\|_{1,s,\Omega}$,  $m$, and $\omega_0$.
	Next, we distinguish the cases $p>2$ and $R>0$ \textit{``sufficiently~small''}:
	
	\textit{Case $p>2$.} In the case $p>2$, the growth properties in \eqref{prop:fem:existence.1} dominate the growth properties in \eqref{prop:fem:existence.2}, so that $A_h\colon V_{h,0}\to (V_{h,0})^*$ is even \textit{``globally''} coercive, i.e., $\lim_{\|\D{z}_h\|_{p,\Omega}\to \infty}{	\langle A_h\vec{z}_h,\vec{z}_h\rangle_{V_{h,0}}/\|\D{z}_h\|_{p,\Omega}}=\infty$.
	
	\textit{Case $R>0$ ``sufficiently small''.}
	It holds that~$\|\vec{g}_h\|_{1,s,\Omega}\to 0$~as~$R\to 0$. Therefore,
	there exists a constant $R>0$ sufficiently small, so that
	\begin{align*}
		c_1\geq 2 \,(c_3\,\|\vec{g}_h\|_{1,s,\Omega})^{p-1} (c_2+c_4+c_5\,\|\vec{f}\|_{V\d})^{2-p}\,.
	\end{align*}
	Thus, $A_h\colon \hspace*{-0.1em}V_{h,0} \hspace*{-0.1em}\to \hspace*{-0.1em} (V_{h,0})^*$ is locally coercive with ${r  \hspace*{-0.1em}\coloneqq \hspace*{-0.1em} (2 (c_2+c_4+c_5\,\|\vec{f}\|_{V\d})c_1^{-1})^{1/(p-1)}}$, since from  $\|\D{z}_h\|_{p,\Omega} = r$, it follows that (cf.\ \cite[Thm.~4.10]{JR23_inhom})
	\begin{align*}
		&(\vec{S}(\D{z}_h+\D{g}_h),\D{z}_h)_\Omega - b(\vec{z}_h+\vec{g}_h, \vec{z}_h+\vec{g}_h,\vec{z}_h) -(\vec{f}, \vec{z}_h)_\Omega \\
		&\geq c_1 \, r^p - c_3\,\|\vec{g}_h\|_{1,s,\Omega}\, r^2 - (c_2+c_4+c_5\,\|\vec{f}\|_{V\d})\, r \\
		&\geq r^p\, [ (\tfrac{c_1}{2} - c_3\,\|\vec{g}_h\|_{1,s,\Omega} \, r^{2-p}) + (\tfrac{c_1}{2} - (c_2+c_4+c_5\,\|\vec{f}\|_{V\d}) \,r^{1-p}) ] \geq 0\,.
	\end{align*}
	As a consequence, the main theorem on locally coercive operators (cf. \cite[Thm. 27.B]{Zeidler2B}) yields the existence of a solution $\vec{u}_h\in V_{h,0}$ to Problem (\hyperlink{Ph}{P$_h$}) with $\norm{\D{u}_h}_p \leq r$.
	Since, by the ansatz \eqref{eq:discrete_ansatz} and by the discrete inf-sup stability result (cf. Lemma~\ref{lem:fem:discrete-inf-sup}), Problem (\hyperlink{Ph}{P$_h$}) and Problem (\hyperlink{Qh}{Q$_h$}) are equivalent, we infer the existence of a solution $(\vec{v}_h,q_h)^\top\in X_h\times Q_h$ to Problem (\hyperlink{Qh}{Q$_h$}) with \eqref{eq:fem:apriori}.\enlargethispage{5mm}
\end{proof}

\begin{theorem} \label{prop:fem:convergence}
	Let the assumptions from Proposition~\ref{prop:fem:existence} be satisfied. If $p \leq \tfrac{3d}{d+2}$, then let,  in addition,  Assumption~\ref{assum:fem:localbasis} be satisfied. Moreover, let $(h_n)_{n\in \mathbb{N}}\subseteq (0,1]$ be a sequence such that $h_n\to 0$ $(n\to \infty)$ and  
	let $(\vec{v}_{h_n}, q_{h_n})^\top\in X_{h_n}\times Q_{h_n}$, $n\in \mathbb{N}$, be the corresponding sequence of solutions to Problem (\hyperlink{Qh}{Q$_{h_n}$}) which fulfill a uniform a priori estimate \eqref{eq:fem:apriori}. Then, there exists a subsequence $(n_k)_{k\in \mathbb{N}}\subseteq \mathbb{N}$ such that
	\begin{align*}
		\begin{aligned}
			\vec{v}_{h_{n_k}} &\weakto \vec{v} &&\quad\textrm{ in } X&&\quad(k\to \infty )\,, \\
			q_{h_{n_k}} &\weakto q &&\quad\textrm{ in } Q^s&&\quad(k\to \infty )\,,
		\end{aligned}
	\end{align*}
	where $(\vec{v}, q)^\top\in X\times Q^s$ is a solution to Problem (\hyperlink{Q}{Q}).
\end{theorem}
\begin{proof}
	Let $(\vec{g}_{h_n})_{n\in \mathbb{N}}\subseteq X_{h_n}$ be a sequence of discrete lifts generated by Lemma~\ref{lem:fem:gh}. Using the ansatz $\vec{u}_{h_n} \coloneqq \vec{v}_{h_n} - \vec{g}_{h_n} \in V_{h,0}$ (cf.\ \eqref{eq:discrete_ansatz}) for all $n\in \mathbb{N}$, 
	due to Lemma~\ref{lem:fem:gh}, the a priori estimate \eqref{eq:fem:apriori}, and the Rellich--Kondrachov compactness theorem for some $\sigma \in ( \min \{1, p'/2\}, p\d/2)$, we deduce the existence of a subsequence $(n_k)_{k\in \mathbb{N}}\subseteq \mathbb{N}$ as well as weak limits $\vec{u}\in V$ and $(\vec{v}, q)^\top\in X\times Q^s$ such that\enlargethispage{5mm}
	\begin{alignat}{3} \label{eq:fem:lt_conv_g}
		\vec{g}_{h_{n_k}} &\to \vec{g} &&\quad \textrm{ in } X^s&&\quad(k\to \infty)\,, \\
		\vec{v}_{h_{n_k}} &\weakto \vec{v} &&\quad \textrm{ in } X&&\quad(k\to \infty)\,, \label{eq:fem:lt_conv_v} \\
		\vec{v}_{h_{n_k}} &\to \vec{v} &&\quad\textrm{ in } L^{2\sigma}(\Omega)&&\quad (k\to \infty)\,,
		\label{eq:fem:lt_conv_v2}\\
		\vec{u}_{h_{n_k}} &\weakto \vec{u}  &&\quad \textrm{ in } V&&\quad(k\to \infty )\,,\label{eq:fem:lt_conv_u} \\
		q_{h_{n_k}} &\weakto q &&\quad \textrm{ in } Q^s&&\quad(k\to \infty )\,. \label{eq:fem:lt_conv_pi}
	\end{alignat}
	In favor of readability, for the remainder of the proof, we assume that $n_k\hspace*{-0.1em}=\hspace*{-0.1em}k$~for~all~$k\hspace*{-0.1em}\in\hspace*{-0.1em} \mathbb{N}$.
	Due to the ansatz $\vec{u}_{h_n} \coloneqq \vec{v}_{h_n} - \vec{g}_{h_n} \in V_{h,0}$ (cf.\ \eqref{eq:discrete_ansatz}) for all $n\in \mathbb{N}$, \eqref{eq:fem:lt_conv_g}, \eqref{eq:fem:lt_conv_u}, and the uniqueness of weak limits, we have that $\vec{v}-\vec{g}=\vec{u}\in V$.
	
	It \hspace*{-0.1mm}is \hspace*{-0.1mm}left \hspace*{-0.1mm}to \hspace*{-0.1mm}verify \hspace*{-0.1mm}that \hspace*{-0.1mm}$(\vec{v}, q)^\top\hspace*{-0.25em}\in\hspace*{-0.15em} X\times Q^s\hspace*{-0.15em}$  \hspace*{-0.1mm}solves \hspace*{-0.1mm}Problem \hspace*{-0.1mm}(\hyperlink{Q}{Q}), \hspace*{-0.1mm}i.e., \hspace*{-0.1mm}satisfies~\hspace*{-0.1mm}\mbox{\eqref{eq:fem:q_main}--\eqref{eq:fem:q_bdry}}:
	
	\textit{ad \eqref{eq:fem:q_bdry}.} From $\vec{v} - \vec{g} \in V$, it follows that $\vec{v} =\vec{g}$ in $W^{1-\frac{1}{s},s}(\partial\Omega)$, i.e., \eqref{eq:fem:q_bdry}.
	
	\textit{ad \eqref{eq:fem:q_div}.} Using \eqref{eq:fem:lt_conv_v}, we find that $\vec{u}\in V_0$ and, thus, that $\vec{v}\in X$ satisfies \eqref{eq:fem:q_div}. In fact, using \eqref{eq:fem:lt_conv_v}, $\smash{\vec{u}_{h_n}}\in  \smash{V_{h_n,0}}$ for all $n\in \mathbb{N}$, \eqref{eq:fem:lt_conv_u}, and Lemma~\ref{lem:fem:appr_y}, for~every~${\eta\in  Y^p}$, we deduce that
	\begin{align*}
		0 = (\div \vec{u}_{h_n}, \Pi_{h_n}^Y \eta)_\Omega \to (\div \vec{u}, \eta)_\Omega \quad (n\to \infty)\,,
	\end{align*}
	i.e., $\div \vec{u}=0$ a.e.\ in $\Omega$ and, thus, $\div\vec{v}=g_1$ in $L^s(\Omega)$, i.e., \eqref{eq:fem:q_div}.

	\textit{ad \eqref{eq:fem:q_main}.} 
	\!The boundedness property of the extra stress tensor $\vec{S}\colon \mathbb{R}^{d\times d}\to \mathbb{R}^{d\times d}_{\textup{sym}}$~(cf. Assumption \ref{assum:extra_stress}\eqref{assum:extra_stress.2}) implies that the sequence $(\vec{S}(\boldsymbol{D}\vec{v}_{h_n}))_{n\in\mathbb{N}}\subseteq L^{p'}(\Omega)$~is~bounded. By passing to a not relabeled subsequence, we obtain a limit $ \vec{S'}\in L^{p'}(\Omega)$ such that 
	\begin{align} \label{eq:fem:lt_conv_s}
		\vec{S}(\boldsymbol{D}\vec{v}_{h_n}) \weakto \vec{S'} \qquad\textrm{in } L^{\smash{p'}}(\Omega)\quad (n\to \infty)\,.
	\end{align} 
	Let $\vec{z} \in V^s$ be fixed, but arbitrary. Due to Corollary~\ref{cor:fem:appr_x}, we have that
	$\Pi_{h_n}^X \vec{z} \to \vec{z}$ in $X^s$ $(n\to \infty)$. Using this and the convergences \eqref{eq:fem:lt_conv_s}, \eqref{eq:fem:lt_conv_v}, \eqref{eq:fem:lt_conv_v2}, and \eqref{eq:fem:lt_conv_pi}, testing \eqref{eq:fem:qh_main} with $\vec{z}_{h_n}=\Pi_{h_n}^X \vec{z}\in X_{h_n}$ for all $n\in \mathbb{N}$, and, subsequently, passing~for~$n\to \infty$,  for every $\vec{z} \in V^s$, we arrive that
	\begin{align} \label{prop:fem:convergence.1}
		(\vec{S'}, \D{z})_\Omega  - (\vec{v} \otimes \vec{v}, \nabla\vec{z}) _\Omega - (q, \div \vec{z})_\Omega  = (\vec{f}, \vec{z})_\Omega \,.
	\end{align}
	It remains to prove that $\vec{S'} = \vec{S}(\D{v})$ in $L^{\smash{p'}}(\Omega)$. To this end, we
	deploy the celebrated discrete (divergence-corrected) Lipschitz truncation method (cf.~\cite{DKS13a} or \cite[Ass. 2.20]{Tscherpel_phd}):
	
	We set $H_{h_n} \coloneqq (\vec{S}(\D{v}_{h_n}) - \vec{S}(\D{v})) : (\D{v}_{h_n} - \D{v})\in L^1(\Omega)$ for all $n\in \mathbb{N}$ and note that $H_h \geq 0$ a.e.\ in  $\Omega$ (since $\vec{S}\colon \mathbb{R}^{d\times d}\to \mathbb{R}^{d\times d}_\textup{sym}$ is monotone, cf.~Assumption~\ref{assum:extra_stress}\eqref{assum:extra_stress.1}). 
	We consider the error $\vec{e}_{h_n}\coloneqq \vec{v}_{h_n} - \vec{v} \in X$ for all $n\in \mathbb{N}$. 
	The properties of  $\Pi_{h_n}^X\colon X\to X_{h_n}$, $n\in \mathbb{N}$, (cf. Assumption~\ref{assum:fem:projection_x})  yield that $\Pi_{h_n}^X (\vec{u}_{h_n}-\vec{u}) \in V_{h_n,0}$ for all $n\in \mathbb{N}$.
	From \eqref{eq:fem:lt_conv_u} and the stability of $\Pi_{h_n}^X\colon X\to X_{h_n}$, $n\in \mathbb{N}$,  (cf. Assumption~\ref{assum:fem:projection_x}), we conclude that $\Pi_{h_n}^X (\vec{u}_{h_n}-\vec{u}) \weakto \vec{0}$ in $V$ $(n\to \infty)$. Then, for every $j ,n\in \N$, let $\vec{v}^{h_n,j} \in V_{h_n}$ and $\vec{w}^{h_n,j} \in V_{h_n,0}$ be the discrete (divergence-corrected) Lipschitz truncations of the sequence $\Pi_{h_n}^X (\vec{u}_{h_n}-\vec{u})\in V_{h_n,0}$, $n\in \mathbb{N}$, as provided by Theorem~\ref{thm:fem:dlt} and Lemma~\ref{lem:fem:dlt}. Let the sequence of bad sets $B_{h_n,j} \subseteq\Omega$, $j,n\in \mathbb{N}$, be defined as in Theorem~\ref{thm:fem:dlt}. Then, we denote the corresponding complements in $\Omega$ as $B_{h_n,j}^c \coloneqq \Omega \setminus B_{h_n,j}$ for all $j,n\in \mathbb{N}$.
	
	Using Cauchy--Schwarz' inequality, for every $n,j\in \mathbb{N}$, we obtain
	\begin{align}\label{prop:fem:convergence.2}
		\| H_{h_n} ^{1/2}\|_{1,\Omega}
		\leq\| H_{h_n} ^{1/2}  \|_{1,B_{h_n,j}}
		+ \|H_{h_n} \|_{1,B_{h_n,j}^c}^{1/2}\,.
	\end{align}
	For the first term on the right-side of \eqref{prop:fem:convergence.2}, due to Cauchy--Schwarz' inequality, (\hyperlink{eq:fem:ltb}{A.2}), $\lambda_{n,j} \geq 1$~for~all~${j,n\in \mathbb{N}}$, \eqref{eq:fem:apriori}, and \eqref{eq:fem:lt_conv_s}, for every $n,j\in \mathbb{N}$, we find that
	\begin{align}\label{prop:fem:convergence.3}
		\| H_{h_n} ^{1/2}  \|_{1,B_{h_n,j}}
		\leq \vert B_{h_n,j}\vert^{1/2} \|H_{h_n} \|_{1,\Omega}^{1/2}
		\leq c\, 2^{-j/2}\,.
	\end{align}
	The second term on the right-side of \eqref{prop:fem:convergence.2}, for every $n,j\in \mathbb{N}$, is decomposed as
	\begin{align}\label{prop:fem:convergence.4}
		\begin{aligned}
			\|H_{h_n}\|_{1,B_{h_n,j}^c}
			&=  (\vec{S}(\D{v}_{h_n} ) - \vec{S}(\D{v}),\vec{D}(\Pi_{h_n} ^X \vec{u} - \vec{u} + \vec{g}_{h_n}  - \vec{g}))_{B_{h_n,j}^c}\\
			&\quad+ (\vec{S}(\D{v}_{h_n} ) - \vec{S}(\D{v}),\vec{D}(\Pi_{h_n} ^X (\vec{u}_{h_n} -\vec{u})))_{B_{h_n,j}^c}
			\\&\eqqcolon I_n^1+I_n^2\,,
		\end{aligned}
	\end{align}
	so that it is only left to show that $\limsup_{n\to \infty}{\vert I_n^i\vert }\leq c\, 2^{-j/p}$ for all $j\in \mathbb{N}$ and $i=1,2$:
	
	\textit{ad $I_n^1$.}
	Due to \eqref{eq:fem:lt_conv_s}, Corollary~\ref{cor:fem:appr_x}, and Lemma~\ref{lem:fem:gh}, it holds that
	\begin{align}\label{prop:fem:convergence.5}
		I_n^1\to 0\quad (n\to \infty)\,.
	\end{align}
	
	\textit{ad $I_n^2$.} For every $n\in \mathbb{N}$, due to the decomposition 
	\begin{align}\label{prop:fem:convergence.6}
		\begin{aligned}
			I_n^2
			&= (\vec{S}(\D{v}_{h_n} ), \D{w}_{h_n,j})_{\Omega}
			- (\vec{S}(\D{v}), \D{w}_{h_n,j})_{\Omega} \\
			& \quad + (\vec{S}(\D{v}_{h_n} ) - \vec{S}(\D{v}), \D{v}_{h_n,j} - \D{w}_{h_n,j})_{\Omega}
			\\&\quad- (\vec{S}(\D{v}_{h_n} ) - \vec{S}(\D{v}), \D{v}_{h_n,j})_{B_{h_n,j}}\\
			&\eqqcolon  I_n^{21} -  I_n^{22}  +  I_n^{23}-  I_n^{24} \,,
		\end{aligned}
	\end{align}
	we have that $\limsup_{n\to \infty}{I_n^2}\leq c\, 2^{-j/p}$ for all $j\in\mathbb{N}$ if $\limsup_{n\to \infty}{I_n^{2i}}\leq c\, 2^{-j/p}$ for all $j\in \mathbb{N}$ and $i=1,\dots,4$:
	
	\textit{ad $I_n^{21}$.} We resort to \eqref{eq:fem:qh_main} to obtain 	$I_n^{21} = - b(\vec{v}_{h_n}, \vec{v}_{h_n}, \vec{w}_{h_n,j}) + (\vec{f}, \vec{w}_{h_n,j})_\Omega$ for all $j,n\in \mathbb{N}$, 
	which, by the same argumentation as above when passing for $h\to 0$~in~\eqref{eq:fem:qh_main}, 
	lets us conclude that 
	\begin{align}\label{prop:fem:convergence.7}
		I_n^{21}  \to 0\quad (n\to \infty)\,.
	\end{align}
	
	\textit{ad $I_n^{22}$.}  Due to (\hyperlink{eq:fem:ltw1s}{A.8}), we have that\enlargethispage{3mm}
	\begin{align}\label{prop:fem:convergence.8}
		I_n^{22}  \to 0\quad (n\to \infty)\,.
	\end{align}
	
	\textit{ad $I_n^{23}$/$I_n^{24}$.} 
	Using  \eqref{eq:fem:lt_conv_s}, (\hyperlink{eq:fem:ltw1s}{A.6}), (\hyperlink{eq:fem:ltw1s}{A.3}), and (\hyperlink{eq:fem:ltw1s}{A.2}), for every $j,n\in \mathbb{N}$, we find that
	\begin{align}\label{prop:fem:convergence.9}
		\vert I_n^{23}\vert  + \vert I_n^{23}\vert \leq c\, 2^{-j/p}\,.
	\end{align}
	As a result, combining \eqref{prop:fem:convergence.7}--\eqref{prop:fem:convergence.9} in \eqref{prop:fem:convergence.6}, for every $j\in \mathbb{N}$, we conclude that
	\begin{align*}
		\limsup_{n\to \infty}{\vert I_n^2\vert} \leq c\, 2^{-j/p}\,,
	\end{align*}
	which together with \eqref{prop:fem:convergence.5} in \eqref{prop:fem:convergence.4} and \eqref{prop:fem:convergence.3} in \eqref{prop:fem:convergence.2}, for every $j\in \mathbb{N}$, 
	implies that
	\begin{align}\label{prop:fem:convergence.10}
		\limsup_{n\to \infty} {\|H_{h_n}^{1/2}  \|_{1,\Omega}}\leq c\,2^{-j/p}\,.
	\end{align}
	Eventually, using \eqref{prop:fem:convergence.10} together with Lemma~\ref{lem:growth_SF}, we conclude that
	\begin{align*}
		\limsup_{n\to \infty} {\| \phi_{\abs{\D{v}}}(\abs{\D{v}_{h_n}-\D{v}})^{1/2} \|_{1,\Omega}}
		\leq \limsup_{n\to \infty} {\|H_{h_n}^{1/2}  \|_{1,\Omega}}
		= c\,2^{-j/p}\to 0\quad (j\to \infty)\,.
	\end{align*}
	Hence, there exists a not relabeled subsequence such that $\phi_{\abs{\D{v}}}(\abs{\D{v}_{h_n}-\D{v}})\to 0$ $(n\to \infty)$ a.e.\ in $\Omega$. Consequently, we have that $\D{v}_{h_n} \to \D{v}$ and $\vec{S}(\D{v}_{h_n}) \to \vec{S}(\D{v})$ $(n\to \infty)$ a.e.\ in $\Omega$. This, in combination with \eqref{eq:fem:lt_conv_s}, yields that $\vec{S'} = \vec{S}(\D{v})$~in~$L^{\smash{p'}}(\Omega)$.
\end{proof}

\begin{remark}
	Note that in the case of homogeneous Dirichlet boundary conditions and a homogeneous divergence constraint weak convergence of FE discretizations has been established for the whole range $p>\frac{2d}{d+2}$ using either exactly \mbox{divergence-free}~FE spaces~(cf.\ \hspace*{-0.1mm}\cite{DKS13a}), reconstruction operators~(cf.\ \hspace*{-0.1mm}\cite{FGS22_boussinesq}) or a regularization~argument~(cf.~\hspace*{-0.1mm}\cite{Tscherpel_phd}).
\end{remark}

%% file: section3.tex
\section{Error estimates}\label{sec:error}

In this section, we derive error estimates for the approximation of a \textit{``regular''} solution $(\vec{v}, q)^\top\in X\times Q^s$ of Problem (\hyperlink{Q}{Q}) through a solution $(\vec{v}_h, q_h)^\top\in X_h\times Q_h$ of Problem (\hyperlink{Qh}{Q$_h$}).

\begin{assum}[regularity] \label{assum:fem:regularity}
	We assume that $(\vec{v}, q)^\top\in X\times Q^s$ is a solution of Problem (\hyperlink{Q}{Q}) that satisfies both $\vec{F}(\D{v}) \in W^{1,2}(\Omega)$ and $q \in W^{1,p'}(\Omega)$. 
\end{assum}

We set $r \coloneqq \min \{2, p\}$ and note that Assumption \ref{assum:fem:regularity} implies that $\vec{v} \in W^{2,r}(\Omega)$, at least if $\delta > 0$ (cf.~\cite{NW05_interior, BDR10_strongsol}). Furthermore, a Sobolev embedding shows that $\vec{v} \in L^{\infty}(\Omega)$.
The regularity presumed above has been proved in some cases (cf. \cite{NW05_interior, BDR10_strongsol}). However, there are various open questions, in particular, on the interplay of the regularity of the pressure and the velocity vector field. We~refer~to~\cite{BBDR12_fem} for a discussion of this issue.\enlargethispage{10mm}

\begin{theorem}[velocity error estimate] \label{prop:fem:error}
	Let the  assumptions of Theorem~\ref{prop:fem:convergence} and Assumption~\ref{assum:fem:regularity} be satisfied with $\delta > 0$. Then, there exists a constant $c_0>0$, depending only on the characteristics of $\vec{S}$, $\delta^{-1}$, $\omega_0$, $m$, $k$, and $\Omega$, such that from 
	\begin{align} \label{eq:fem:error_smallness}
		\|\vec{v}\|_{1,(r\d/2)',\Omega}
		\leq c_0\,,
	\end{align}
	it follows that
	\begin{align} \label{eq:fem:error_vel}
		\begin{aligned}
			\|\vec{F}(\D{v}_h) - \vec{F}(\D{v})\|_{2,\Omega}^2
			&\leq c\, h^2\,\|\nabla \vec{F}(\D{v})\|_{2,\Omega}^2+ c\, \rho_{(\phi_{\abs{\D{v}}})\d,\Omega}(h \nabla q)\,,
		\end{aligned}
	\end{align}
	where $c > 0$ is a constant depending only on the characteristics of $\vec{S}$, $\delta^{-1}$, $\omega_0$,~$m$,~$k$, $\|\vec{v}\|_{\infty,\Omega}$, $\|\D{v}\|_{p,\Omega}$, $c_0$, and $\Omega$.
\end{theorem}

\begin{proof}
	Since the Problems (\hyperlink{Q}{Q}), (\hyperlink{Qh}{Q$_h$}) do not depend on the particular choice~of~$\vec{g}_b$, we may assume that $\vec{g}_b = \vec{v}$ for the remainder of the proof.
	Then, abbreviating 
	\begin{align*}
		\begin{aligned}
			\vec{e}_h \coloneqq \vec{v}_h - \vec{v} \in X\,,\qquad
			\vec{r}_h \coloneqq \vec{u}_h - \vec{u} = \vec{u}_h \in V\,,
		\end{aligned}
	\end{align*}we arrive at the decomposition
	\begin{align}\label{eq:fem:decomp_eh} 
		\vec{e}_h = \Pi_h^X \!\vec{r}_h + \Pi_h^X \vec{v} - \vec{v}\quad\text{ in }W^{1,p}(\Omega)\,.
	\end{align}
	Due \hspace*{-0.1mm}to \hspace*{-0.1mm}\eqref{eq:fem:decomp_eh}, \hspace*{-0.1mm}\eqref{eq:growth_S}, \hspace*{-0.1mm}\eqref{eq:fem:orliczyoung}, \hspace*{-0.1mm}and \hspace*{-0.1mm}the \hspace*{-0.1mm}approximation \hspace*{-0.1mm}properties \hspace*{-0.1mm}of \hspace*{-0.1mm}$\Pi_h^X$ \hspace*{-0.5mm}(cf.\ \hspace*{-0.1mm}\cite[Thms.~\hspace*{-0.1mm}3.4,~\hspace*{-0.1mm}5.1]{BBDR12_fem}), we have that 
	\begin{align}\begin{aligned} \label{eq:fem:error_1}
			c\, \|\vec{F}(\D{v}_h) - \vec{F}(\D{v})\|_{2,\Omega}^2 
			&\leq (\vec{S}(\D{v}_h) - \vec{S}(\D{v}), \D{e}_h)_\Omega \\
			&\leq (\vec{S}(\D{v}_h) - \vec{S}(\D{v}), \vec{D}\Pi_h^X \!\vec{r}_h)_\Omega 
			\\&\quad +c_{\epsilon}\, \|\vec{F}(\D{v})-\vec{F}(\D{}\Pi_h^X\vec{v})\|_{2,\Omega}^2\\&\quad + \epsilon \,\|\vec{F}(\D{v}_h) - \vec{F}(\D{v})\|_{2,\Omega}^2
			\\&\leq (\vec{S}(\D{v}_h) - \vec{S}(\D{v}), \vec{D}\Pi_h^X \!\vec{r}_h)_\Omega\\& \quad +c_\epsilon \, h^2\, \|\nabla \vec{F}(\D{v})\|_{2,\Omega}^2\\&\quad +
			\epsilon \,\|\vec{F}(\D{v}_h) - \vec{F}(\D{v})\|_{2,\Omega}^2 \,.
	\end{aligned}\end{align}
	Subtracting the first equations in Problem (\hyperlink{Q}{Q}) and Problem (\hyperlink{Qh}{Q$_h$}), for every $\vec{z}_h \in V_h$, yields the error equation
	\begin{align}
		\begin{aligned} \label{eq:fem:error_equation}
			(\vec{S}(\D{v}_h) - \vec{S}(\D{v}), \D{z}_h)_\Omega
			= ((q_h - q)\vec{I}, \D{z}_h)_\Omega + [b(\vec{v}, \vec{v}, \vec{z}_h) - b(\vec{v}_h, \vec{v}_h, \vec{z}_h)]\,.
		\end{aligned}
	\end{align}
	In particular, for  $\vec{z}_h=\Pi_h^X \!\vec{r}_h \in V_h$ in \eqref{eq:fem:error_equation}, we deduce that
	\begin{align}
		\begin{aligned} \label{eq:fem:error_2}
			(\vec{S}(\D{v}_h) - \vec{S}(\D{v}), \vec{D}\Pi_h^X \!\vec{r}_h)_\Omega 
			&= ((q_h - q)\boldsymbol{I}, \boldsymbol{D} \Pi_h^X \!\vec{r}_h)_\Omega \\&\quad + [b(\vec{v}, \vec{v}, \Pi_h^X \!\vec{r}_h) - b(\vec{v}_h, \vec{v}_h, \Pi_h^X \!\vec{r}_h)]
			\\&	\eqqcolon I_h^1 + I_h^2\,.
		\end{aligned}
	\end{align}
	Thus, it is left to estimate the terms $I_h^1$ and $I_h^2$:
	
	\textit{ad $I_h^1$.} Let $\eta_h\hspace*{-0.1em} \in\hspace*{-0.1em} Q_h$ be arbitrary, but fixed. \hspace*{-0.1em}The divergence constraints~\eqref{eq:fem:q_div},~\eqref{eq:fem:qh_div}, the vanishing mean values of $q_h, \eta_h \in Q_h$, the $\epsilon$-Young inequality \eqref{eq:fem:orliczyoung}~with~$\psi=\varphi_{\vert \D{v}\vert}$, \eqref{eq:growth_S}, and the approximation properties of $\Pi_h^X$ (cf.\ \cite[Thms.\ 3.4,  5.1]{BBDR12_fem}) yield that
	\begin{align}
		\begin{aligned} \label{eq:fem:error_2.1}
			\vert I_h^1\vert 
			&= \vert ((\eta_h - q)\boldsymbol{I},  \boldsymbol{D} ( \vec{v} - \Pi_h^X \vec{v} ) )_\Omega+((\eta_h - q)\boldsymbol{I},  \boldsymbol{D}\vec{e}_h)_\Omega \vert  \\
			&\leq c_{\epsilon}\,  \rho_{(\phi_{\abs{\D{v}}})\d,\Omega}(\eta_h - q)+c\,  \|\vec{F}(\D{v}) - \vec{F}(\vec{D}\Pi_h^X \vec{v} )\|_{2,\Omega}^2
			\\&\quad+ \epsilon\,   c\,  \|\vec{F}(\D{v}_h) - \vec{F}(\D{v})\|_{2,\Omega}^2
			\\	&\leq c_{\epsilon}\,  \rho_{(\phi_{\abs{\D{v}}})\d,\Omega}(\eta_h - q)+c \, h^2\,  \|\nabla \vec{F}(\D{v})\|_{2,\Omega}^2
			\\&\quad+ \epsilon\,   c\,  \|\vec{F}(\D{v}_h) - \vec{F}(\D{v})\|_{2,\Omega}^2\,.
		\end{aligned}
	\end{align}
	For $\eta_h \coloneqq \Pi_h^Y q - \mean{\Pi_h^Y q}_{\Omega} \in Q_h$ in \eqref{eq:fem:error_2.1}, using Lemma \ref{lem:fem:comparable_qh} and the approximation~properties of $\Pi_h^Y$ (cf.\ \cite[Lem.\ 5.4]{BBDR12_fem}), we obtain
	\begin{align}\label{eq:fem:appr_pressure}
		\begin{aligned}
			\vert I_h^1\vert 
			&\leq c_{\epsilon} \, \rho_{(\phi_{\abs{\D{v}}})\d,\Omega}(h \nabla q) + c\,  h^2 \, \|\nabla \vec{F}(\D{v})\|_{2,\Omega}^2
			\\&\quad+ \epsilon\, c \,  \|\vec{F}(\D{v}_h) - \vec{F}(\D{v})\|_{2,\Omega}^2\,.
		\end{aligned}
	\end{align}
	Before we move on to estimate Temam's modification, we note that  
	Poincaré's inequality, Korn's inequality, the $W^{2,r}$-approximation properties of $\Pi_h^X$ (cf.\ \cite[Thm.~4.6]{DR07_interpolation}), and that $\|\nabla^2 \vec{v}\|_{r,\Omega}\leq \|\nabla\vec{F}(\nabla \vec{v})\|_{2,\Omega}^2(1+\|\D{v}\|_{r,\Omega})^{2-r}$ (cf. \cite[Lem.\ 4.5]{BDR10_strongsol}) yield that\enlargethispage{2mm}
	\begin{align}\label{eq:fem:eh_1p}
		\begin{aligned}
			\norm{\vec{e}_h}_{1,r,\Omega}
			&\leq c\, \|\D{r}_h\|_{r,\Omega} + \|{\Pi_h^X} \vec{v} - \vec{v}\|_{1,r,\Omega}\\
			&\leq c\,\|\D{e}_h\|_{r,\Omega} + c \,h \,\|\nabla^2 \vec{v}\|_{r,\Omega} \\
			&\leq c\,\|\D{e}_h\|_{r,\Omega} + c \,h \,\|\nabla\vec{F}(\nabla \vec{v})\|_{2,\Omega}(1+\|\D{v}\|_{r,\Omega})^{\smash{\frac{2-r}{2}}}\,.
		\end{aligned}
	\end{align}
	
	\textit{ad $I_h^2$.}
	Using that $\Pi_h^X \!\vec{r}_h = \Pi_h^X\!\vec{e}_h$ and the skew-symmetry of $\tilde{b}\colon (X)^2\to X^s\to \mathbb{R}$, we arrive at the decomposition
	\begin{align}\label{eq:fem:error_4}
		\begin{aligned}
			I_h^2
			&= \tilde{b}(\vec{v}, \vec{v} - \Pi_h^X \vec{v}, \Pi_h^X\!\vec{e}_h)
			- \tilde{b}(\vec{e}_h, \Pi_h^X \vec{v}, \Pi_h^X\!\vec{e}_h)
			- \tfrac{1}{2} (g_1 \vec{e}_h, \Pi_h^X\!\vec{e}_h)
			\\&\eqqcolon I_h^{21}+ I_h^{22}+ I_h^{23}\,.
		\end{aligned}
	\end{align}
	Thus,  it is left to estimate the terms $I_h^{21}$, $I_h^{22}$, and $I_h^{23}$:
	
	\textit{ad $I_h^{21}$.}
	The regularity Assumption \ref{assum:fem:regularity} and $r\ge (r^*)'$ (due to $p\ge \frac{2d}{d+1}$) yield~that $\vec{v} \in W^{2,(r^*)'}(\Omega)\cap L^\infty(\Omega)$  (cf. \cite[Lem.\ 2.6]{KR23_ldg2}, \cite[Lem.\ 4.5]{BDR10_strongsol}). Using  Hölder's inequality, Sobolev's embedding theorem, 
	the $W^{2,(r\d)'}$-approximation property of $\Pi_h^X$ (cf.\ \cite[Thm.\ 4.6]{DR07_interpolation}),
	the $W^{1,r}$-stability~of $\Pi_h^X$ (cf.\ \cite[Thm.\ 3.2]{BBDR12_fem}), \eqref{eq:fem:eh_1p}, that $ \|\nabla^2\vec{v}\|_{2r/(2-p+r),\Omega} ^2\leq c\,\|\nabla \vec{F}(\D{v})\|_{2,\Omega}^2(\delta+\|\D{v}\|_{r,\Omega})^{2-p}$ (cf. \cite[Lem.\ 4.5]{BDR10_strongsol}) together with $(r^*)'\leq 2r/(2-p+r)$ if $p\hspace*{-0.15em}<\hspace*{-0.15em}2$ \hspace*{-0.1mm}(due \hspace*{-0.1mm}to \hspace*{-0.1mm}$p\hspace*{-0.15em}\ge\hspace*{-0.15em} \frac{2d}{d+1}$), \hspace*{-0.1mm}that \hspace*{-0.1mm}$ \|\nabla^2\vec{v}\|_{2,\Omega} ^2\hspace*{-0.15em}\leq\hspace*{-0.15em} c\,\|\nabla \vec{F}(\D{v})\|_{2,\Omega}^2$ \hspace*{-0.1mm}together~\hspace*{-0.1mm}with~\hspace*{-0.1mm}${(r^*)'\hspace*{-0.15em}\leq\hspace*{-0.15em} 2}$~\hspace*{-0.1mm}if~\hspace*{-0.1mm}${p\hspace*{-0.15em}\ge\hspace*{-0.15em} 2}$, and the $\epsilon$-Young inequalities~\eqref{eq:fem:orliczyoung}~with~${\psi=\vert \cdot\vert^2}$,~we~get\enlargethispage{5mm}
	\begin{align}
		\label{eq:fem:error_5}
		\begin{aligned}
			\vert I_h^{21}\vert 
			&\leq c\, \|\vec{v}\|_{\infty,\Omega}  \|\vec{v} - \Pi_h^X \vec{v}\|_{1,(r\d)',\Omega} \|\Pi_h^X\vec{e}_h\|_{r\d,\Omega} \\
			&\leq c\, h\,\|\vec{v}\|_{\infty,\Omega}  \|\nabla^2\vec{v}\|_{(r\d)',\Omega} \|\vec{e}_h\|_{1,r,\Omega}\\
			&\leq c_{\epsilon}\,h^2\,\|\vec{v}\|_{\infty,\Omega} ^2\,\|\nabla \vec{F}(\D{v})\|_{2,\Omega}^2(\delta+\|\D{v}\|_{r,\Omega})^{2-p}
			\\&\quad+ \epsilon\, c\,(\|\D{e}_h\|_{r,\Omega}^2+h^2\,\|\nabla\vec{F}(\nabla \vec{v})\|_{2,\Omega}^2(1+\|\D{v}\|_{r,\Omega})^{2-r})\,.
		\end{aligned}
	\end{align}
	
	\textit{ad $I_h^{22}$.} Using  Hölder's inequality, the $W^{1,(r\d/2)'}$-stability of $\Pi_h^X$ (cf.\ \cite[Thm.\ 3.2]{BBDR12_fem}), Sobolev's embedding theorem, and \eqref{eq:fem:eh_1p}, we deduce that
	\begin{align}
		\label{eq:fem:error_6}
		\begin{aligned}
			\vert I_h^{22}\vert 
			&	\leq c\, \|\vec{e}_h\|_{r\d,\Omega} \|\Pi_h^X \vec{v}\|_{(r\d/2)',\Omega}\|\Pi_h^X \vec{e}_h\|_{1,r,\Omega}\\
			&	\leq c\, \|\vec{v}\|_{1,(r\d/2)',\Omega}\|\vec{e}_h\|_{1,r,\Omega}^2
			\\&	\leq c \,\|\vec{v}\|_{1,(r\d/2)',\Omega} (\|\D{e}_h\|_{r,\Omega}^2 + h^2\, \|\nabla\vec{F}(\nabla \vec{v})\|_{2,\Omega}^2(1+\|\D{v}\|_{r,\Omega})^{2-r})\,.
		\end{aligned}
	\end{align}

	\textit{ad $I_h^{23}$.} Using  Hölder's inequality, Sobolev's embedding theorem, that $\textup{div}\,\vec{v}=g_1$ in $L^p(\Omega)$, and \eqref{eq:fem:eh_1p}, we find that
	\begin{align}
		\label{eq:fem:error_7}
		\begin{aligned}
			\vert I_h^{23}\vert &\leq 
			c\, \|g_1\|_{(r\d/2)',\Omega}\|\vec{e}_h\|_{r\d,\Omega}\|\Pi_h^X\vec{e}_h\|_{r\d,\Omega}\\
			&	\leq c\,  \|\vec{v}\|_{1,(r\d/2)',\Omega}\|\vec{e}_h\|_{1,r,\Omega}^2
			\\
			&	\leq c\,  \|\vec{v}\|_{1,(r\d/2)',\Omega}(\|\D{e}_h\|_{r,\Omega}^2 + h^2 \,\|\nabla\vec{F}(\nabla \vec{v})\|_{2,\Omega}^2(1+\|\D{v}\|_{r,\Omega})^{2-r})\,.
		\end{aligned}
	\end{align}
	As a result, combining  \eqref{eq:fem:error_5}--\eqref{eq:fem:error_7} in \eqref{eq:fem:error_4}, we arrive at
	\begin{align}
		\label{eq:fem:error_8}
		\begin{aligned}
			\vert I_h^2\vert& \leq  c_{\epsilon}\,h^2\,\|\vec{v}\|_{\infty,\Omega} ^2\,\|\nabla \vec{F}(\D{v})\|_{2,\Omega}^2(\delta+\|\D{v}\|_{r,\Omega})^{2-p} \\&\quad +c\,\|\vec{v}\|_{1,(r\d/2)',\Omega}\,h^2\,\|\nabla\vec{F}(\nabla \vec{v})\|_{2,\Omega}^2(1+\|\D{v}\|_{r,\Omega})^{2-r}\\&\quad+(\varepsilon+\|\vec{v}\|_{1,(r\d/2)',\Omega} )\,\|\D{e}_h\|_{r,\Omega}^2 \,.
		\end{aligned}
	\end{align}
	Due to \cite[Lem.\ 4.1]{BDR10_strongsol}, 	if $p \leq 2$, then  we have that
	\begin{align}\label{eq:fem:error_9}
		\begin{aligned}
			\|\D{e}_h\|_{p,\Omega}^2
			&\leq c\, \|\vec{F}(\D{v}_h) - \vec{F}(\D{v})\|_{2,\Omega}^2\|\delta + \vert\D{v}\vert + \vert\D{v}_h\vert \|_{p,\Omega}^{2-p} \\
			&\leq c \,\|\vec{F}(\D{v}_h) - \vec{F}(\D{v})\|_{2,\Omega}^2 \|\delta +\vert\D{v}\vert\|_{p,\Omega}^{2-p}\,.
		\end{aligned}
	\end{align}
	Similar to \cite[Eq. (4.50)]{KR23_ldg2},   if $p > 2$, then we have that 
	\begin{align}\label{eq:fem:error_10}
		\begin{aligned}
			\|\D{e}_h\|_{2,\Omega}^2
			\leq c\, \delta^{2-p}\, \|\vec{F}(\D{v}_h) - \vec{F}(\D{v})\|_{2,\Omega}^2\,.
		\end{aligned}
	\end{align}
	Thus, combining \eqref{eq:fem:error_9} and \eqref{eq:fem:error_10}, in any case, we have that
	\begin{align} \label{eq:fem:error_11}
		\|\D{e}_h\|_{r,\Omega}^2
		\leq (\max\{\delta, \|\delta +\vert\D{v}\vert\|_{p,\Omega}\})^{2-p}\, \|\vec{F}(\D{v}_h) - \vec{F}(\D{v})\|_{2,\Omega}^2\,.
	\end{align}
	Putting everything together, if we insert \eqref{eq:fem:appr_pressure} and \eqref{eq:fem:error_8} in \eqref{eq:fem:error_2}, using \eqref{eq:fem:error_11}~in~doing~so, we deduce that
	\begin{align}
			(\vec{S}(\D{v}_h) - \vec{S}(\D{v}), \vec{D}\Pi_h^X \!\vec{r}_h)_\Omega 
			&\leq c_{\epsilon} \, (\rho_{(\phi_{\abs{\D{v}}})\d,\Omega}(h \nabla q)+h^2\, \|\nabla \vec{F}(\D{v})\|_{2,\Omega}^2\notag \\
			&\quad+c\,(\varepsilon+\|\vec{v}\|_{1,(r\d/2)',\Omega})\, \|\vec{F}(\D{v}_h) - \vec{F}(\D{v})\|_{2,\Omega}^2\,.	\label{eq:fem:error_12}
	\end{align}
	Then, if we insert \eqref{eq:fem:error_12} into \eqref{eq:fem:error_1},  we obtain
	\begin{align*}
		\|\vec{F}(\D{v}_h) - \vec{F}(\D{v})\|_{2,\Omega}^2 
		&\leq c_{\epsilon} \,h^2\, \|\nabla \vec{F}(\D{v})\|_{2,\Omega}^2
		+ c_{\epsilon} \, (\rho_{(\phi_{\abs{\D{v}}})\d,\Omega}(h \nabla q)
		\\&\quad+c\,(\varepsilon+\|\vec{v}\|_{1,(r\d/2)',\Omega}) \|\vec{F}(\D{v}_h) - \vec{F}(\D{v})\|_{2,\Omega}^2\,.
	\end{align*}
	Eventually, choosing $\varepsilon>0$ small enough and a constant $c_0>0$ small enough,~from~\eqref{eq:fem:error_smallness}, it follows the claimed a priori error estimate for the velocity vector field \eqref{eq:fem:error_vel}.\enlargethispage{7mm}
\end{proof}

Analogously to \cite[Thm.~2.14]{BBDR12_fem} and \cite[Cor.~4.2]{KR23_ldg2}, from Theorem \ref{prop:fem:error}, we derive a priori error estimates for the velocity vector field yielding explicit error decay rates.
\begin{cor} \label{cor:fem:error_vel}
	Let the assumptions of Theorem~\ref{prop:fem:error} be satisfied. Then, it holds that
	\begin{align*}
		\begin{aligned}
			\|\vec{F}(\D{v}_h) - \vec{F}(\D{v})\|_{2,\Omega}^2
			\leq c\, h^2 \|\nabla \vec{F}(\D{v})\|_{2,\Omega}^2 + c\, h^{\min\{2,p'\}}\,\rho_{(\phi_{\abs{\D{v}}})\d,\Omega}(\nabla q)\,,
		\end{aligned}
	\end{align*}
	where $c > 0$ is a constant depending only on the characteristics of $\vec{S}$, $\delta^{-1}$, $\omega_0$,~$m$,~$k$, $\|\vec{v}\|_{\infty,\Omega}$, $\|\D{v}\|_{p,\Omega}$, $c_0$, and $\Omega$.
	If $p > 2$ and, in addition, $\vec{f} \in L^2(\Omega)$, then it holds that
	\begin{align*}
		\begin{aligned}
			\|\vec{F}(\D{v}_h) - \vec{F}(\D{v})\|_{2,\Omega}^2
			\leq c\, h^2\|\nabla \vec{F}(\D{v})\|_{2,\Omega}^2+c\,h^2\,\|(\delta+\vert \D{v}\vert)^{\frac{p-2}{2}}\nabla q\|_{2,\Omega}^2\,,
		\end{aligned}
	\end{align*}
	where $c > 0$ is a constant depending only on the characteristics of $\vec{S}$, $\delta^{-1}$, $\omega_0$,~$m$,~$k$, $\|\vec{v}\|_{\infty,\Omega}$, $\|\D{v}\|_{p,\Omega}$, $c_0$, and $\Omega$.
\end{cor}

\begin{theorem}[pressure error estimate] \label{prop:fem:error_pres}
	Let the assumptions of Theorem~\ref{prop:fem:error} be satisfied. Then, it holds that
	\begin{align} \label{eq:fem:error_pres}
		\begin{aligned}
			\|q_h - q\|_{s',\Omega}
			&\leq c \,h\,(\|\nabla\vec{F}(\nabla \vec{v})\|_{2,\Omega} +\|\nabla q\|_{s'})\\&\quad + c\,\|\vec{F}(\D{v}_h) - \vec{F}(\D{v})\|_{2,\Omega}^{\min \{1, 2/p'\}}\,,
		\end{aligned}
	\end{align}
	where $c > 0$ is a constant depending only on the characteristics of $\vec{S}$, $\delta^{-1}$, $\omega_0$,~$m$,~$k$, $\|\vec{v}\|_{\infty,\Omega}$, $\|\D{v}\|_{p,\Omega}$, $c_0$, and $\Omega$.\enlargethispage{6mm}
\end{theorem}

\begin{proof}
	We choose $\eta_h \coloneqq \Pi_h^Y q - \mean{\Pi_h^Y q}_{\Omega} \in Q_h$ as in the proof of Proposition~\ref{prop:fem:error}. Then, by the triangle inequality, the discrete inf-sup stability result (cf.\ Lemma~\ref{lem:fem:discrete-inf-sup}), and Hölder's inequality, we have that
	\begin{align}\label{prop:fem:error_pres.1}
		\begin{aligned}
			\|q_h - q\|_{s',\Omega}
			&\leq \|q_h - \eta_h\|_{s',\Omega} + \|\eta_h - q\|_{s',\Omega} \\
			&\leq c\, \sup_{\vec{z}_h \in V_h\colon \|\vec{z}_h\|_{1,s} \leq 1} {(q_h - \eta_h, \div \vec{z}_h)_\Omega} + \|\eta_h - q\|_{s',\Omega} \\
			&\leq c\, \sup_{\vec{z}_h \in V_h\colon \|\vec{z}_h\|_{1,s} \leq 1} {(q_h - q, \div \vec{z}_h)_\Omega} + c \,\|\eta_h - q\|_{s',\Omega}\,.
		\end{aligned}
	\end{align}
	Next, let $\vec{z}_h \in V_h$ be such that $\|\vec{z}_h\|_{1,s} \leq 1$. Then, resorting to the error equation~\eqref{eq:fem:error_equation}, we arrive at the decomposition
	\begin{align}\label{prop:fem:error_pres.2}
		\begin{aligned}
			(q_h - q, \div \vec{z}_h)_\Omega 
			&= (\vec{S}(\D{v}_h) - \vec{S}(\D{v}), \D{z}_h)_\Omega
			+ \tfrac{1}{2} (g_1 (\vec{v}_h-\vec{v}), \vec{z}_h) _\Omega
			\\&\quad+ [\tilde{b}(\vec{v}_h, \vec{v}_h, \vec{z}_h) - \tilde{b}(\vec{v}, \vec{v}, \vec{z}_h)] \\
			&\eqqcolon I_h^1 + I_h^2 + I_h^3\,,
		\end{aligned}
	\end{align}
	so that is left to estimate the terms $ I_h^1$, $I_h^2$,  and $I_h^3$:
	
	\textit{ad $I_h^1$.} The shift change \eqref{lem:shift_change.1} and $p \leq s$ yield that $\rho_{\phi_{\abs{\D{v}}},\Omega}(\D{z}_h)
	\leq c \,\|\D{z}_h\|_{p,\Omega}^p + c\, \|\delta + \vert \D{v}\vert \|_{p,\Omega}^p
	\leq c$, 
	which,
	appealing to \cite[Lem.\ 4.8.4]{Pick}, implies that
	\begin{align}\label{prop:fem:error_pres.3}
		\smash{	\|\D{z}_h\|_{\phi_{\abs{\D{v}}},\Omega}
			\leq \max \{1, \rho_{\phi_{\abs{\D{v}}},\Omega}(\D{z}_h) \}
			\leq c\,.}
	\end{align}
	On the other hand, by Lemma~\ref{lem:growth_SF}, we have that
	\begin{align*}
		\smash{	\rho_{(\phi_{\abs{\D{v}}})\d,\Omega} (\vec{S}(\D{v}_h) - \vec{S}(\D{v}))
			\leq c\, \|\vec{F}(\D{v}_h) - \vec{F}(\D{v})\|_{2,\Omega}^2\,,}
	\end{align*}
	which, appealing to \cite[Lem.\ 4.4]{KR23_ldg3}, implies that
	\begin{align}\label{prop:fem:error_pres.4}
		\smash{	\|\vec{S}(\D{v}_h) - \vec{S}(\D{v})\|_{(\phi_{\abs{\D{v}}})\d,\Omega}
			\leq c\,\|\vec{F}(\D{v}_h) - \vec{F}(\D{v})\|_{2,\Omega}^{\min\{1,2/p'\}}\,.}
	\end{align}
	Using the generalized Hölder inequality~\eqref{eq:gen_hoelder}, \eqref{prop:fem:error_pres.3}, and \eqref{prop:fem:error_pres.4}, we conclude that
	\begin{align}\label{prop:fem:error_pres.5}
		\begin{aligned}
			\vert I_h^1\vert 
			&\leq 2 \,\|\vec{S}(\D{v}_h) - \vec{S}(\D{v})\|_{(\phi_{\abs{\D{v}}})\d,\Omega} \|\D{z}_h\|_{\phi_{\abs{\D{v}}},\Omega}
			\\&	\leq c \,\|\vec{F}(\D{v}_h) - \vec{F}(\D{v})\|_{2,\Omega}^{\min \{1, 2/p'\}}\,.
		\end{aligned}
	\end{align}
	
	\textit{ad $I_h^2$.} 
	Using Hölder's inequality, Sobolev's embedding theorem, \eqref{eq:fem:eh_1p}, \eqref{eq:fem:error_11}, and $\norm{\vec{z}_h}_{1,s} \leq 1$ together with $s\ge p$, we obtain
	\begin{align}\label{prop:fem:error_pres.6}
		\begin{aligned}
			\vert I_h^2\vert 
			&\leq c \, \|g_1\|_s\|\vec{e}_h\|_{r\d,\Omega}\|\vec{z}_h\|_{p\d}
			\\&\leq c \, \|\vec{F}(\D{v}_h) - \vec{F}(\D{v})\|_{2,\Omega} + c\, h\,\|\nabla\vec{F}(\nabla \vec{v})\|_{2,\Omega}\,.
		\end{aligned}
	\end{align}
	
	\textit{ad $I_h^3$.} Appealing to the decomposition
	\begin{align}\label{prop:fem:error_pres.7}
		\begin{aligned}
			I_h^3
			&= - \tilde{b}(\vec{v}, \vec{v} - \Pi_h^X \vec{v}, \vec{z}_h)
			+ \tilde{b}(\vec{e}_h, \Pi_h^X \vec{v}, \vec{z}_h)
			+ \tilde{b}(\vec{v}_h, \Pi_h^X\!\vec{e}_h, \vec{z}_h)
			\\&	\eqqcolon I_h^{31}+I_h^{33}+I_h^{33}\,,
		\end{aligned}
	\end{align}
	it is enough to estimate the terms $I_h^{31}$, $I_h^{32}$ and $ I_h^{33}$:
	
	\textit{ad $I_h^{31}$.} 
	Using Hölder's inequality, Sobolev's embedding theorem together with $s^*\ge (r^*)'$, and 
	$W^{2,r}$-approximation properties of $\Pi_h^X$ (cf.\ \cite[Thm.~4.6]{DR07_interpolation}), we find that
	\begin{align}\label{prop:fem:error_pres.8}
		\begin{aligned}
			\vert I_h^{31}\vert 
			&\leq \|\vec{v}\|_{\infty,\Omega}  \|\vec{z}_h\|_{1,(r\d)',\Omega} \|\vec{v} - \Pi_h^X \vec{v}\|_{1,r,\Omega} 
			\\&\leq c \,h\,\|\vec{v}\|_{\infty,\Omega}  \|\vec{z}_h\|_{1,s} \|\nabla^2 \vec{v}\|_{r,\Omega} 
			\\&\leq c\, h\,\|\nabla \vec{F}(\D{v})\|_{2,\Omega}\,.
		\end{aligned}
	\end{align}
	
	\textit{ad $I_h^{32}$.} Using Hölder's inequality together with $s\ge 2(r^*)'$, \eqref{eq:fem:eh_1p}, $\norm{\vec{z}_h}_{1,s} \leq 1$, and the $W^{1,r^*}$-stability properties of  $\Pi_h^X$ (cf.\ \cite[Thm.~3.2]{BBDR12_fem}), we find that
	\begin{align}\label{prop:fem:error_pres.9}
		\begin{aligned}
			\vert I_h^{32}\vert 
			&\leq c\, \|\vec{e}_h\|_{r\d,\Omega} \|\vec{z}_h\|_{1,s} \|\Pi_h^X \vec{v}\|_{1, r\d}
			\\&\leq  c\, \|\vec{F}(\D{v}_h) - \vec{F}(\D{v})\|_{2,\Omega} + c\, h\,\|\nabla\vec{F}(\nabla \vec{v})\|_{2,\Omega}\,.
		\end{aligned}
	\end{align}
	
	\textit{ad $I_h^{33}$.}	Using Hölder's inequality together with $s\ge 2(p^*)'$,  \eqref{eq:fem:apriori}, $\norm{\vec{z}_h}_{1,s} \leq 1$, the $W^{1,p}$-stability properties of  $\Pi_h^X$ (cf.\ \cite[Thm.~3.2]{BBDR12_fem}), and \eqref{eq:fem:eh_1p}, we obtain
	\begin{align}\label{prop:fem:error_pres.10}
		\begin{aligned}
			\vert I_h^{33}\vert 
			&\leq c\, \|\vec{v}_h\|_{p\d} \|\vec{z}_h\|_{1,s} \|\Pi_h^X \vec{e}_h\|_{1,p}
			\\&\leq c\, \|\vec{F}(\D{v}_h) - \vec{F}(\D{v})\|_{2,\Omega} + c\, h\,\|\nabla\vec{F}(\nabla \vec{v})\|_{2,\Omega}\,.
		\end{aligned}
	\end{align}
	Putting everything together, combining \eqref{prop:fem:error_pres.8}--\eqref{prop:fem:error_pres.10} in \eqref{prop:fem:error_pres.7}, we arrive at
	\begin{align}\label{prop:fem:error_pres.11}
		\vert I_h^3\vert \leq c\, \|\vec{F}(\D{v}_h) - \vec{F}(\D{v})\|_{2,\Omega} + c\, h\,\|\nabla\vec{F}(\nabla \vec{v})\|_{2,\Omega}\,.
	\end{align}
	Then, combining \eqref{prop:fem:error_pres.5}, \eqref{prop:fem:error_pres.6}, and \eqref{prop:fem:error_pres.11} in \eqref{prop:fem:error_pres.2} yields that
	\begin{align}\label{prop:fem:error_pres.12}
		\vert (q_h - q, \div \vec{z}_h)_\Omega\vert  \leq c\, \|\vec{F}(\D{v}_h) - \vec{F}(\D{v})\|_{2,\Omega} ^{\min\{2,p'\}}+ c\, h\,\|\nabla\vec{F}(\nabla \vec{v})\|_{2,\Omega}\,.
	\end{align}
	On the other hand, using that $\Pi_h^Y\langle q\rangle_\Omega=\langle q\rangle_\Omega$, stability and approximation properties of $\Pi_h^Y\colon Y\to Y_h$ (cf. \cite[Lem.\ 5.2]{BBDR12_fem}), Hölder's inequality, Poincar\'e's inequality, and that $p\leq s$, i.e., $s'\leq p'$ , recalling that $\eta_h \coloneqq \Pi_h^Y q - \mean{\Pi_h^Y q}_{\Omega} \in Q_h$, we observe that
	\begin{align}\label{prop:fem:error_pres.13}
		\begin{aligned}
			\|\eta_h-q\|_{s',\Omega}&\leq  \|\Pi_h^Y(q-\langle q\rangle_\Omega)\|_{s',\Omega}+\|\langle q- \Pi_h^Yq\rangle_\Omega\|_{s',\Omega}
			\\&\leq c\,  \|q-\langle q\rangle_\Omega\|_{s',\Omega}+c\,\| q- \Pi_h^Yq\|_{s',\Omega}
			\\&\leq c\,  h\,\|\nabla q\|_{s'}\,.
		\end{aligned}
	\end{align}
	Eventually, using \eqref{prop:fem:error_pres.12} and \eqref{prop:fem:error_pres.13} in \eqref{prop:fem:error_pres.1}, we conclude that the claimed a priori error estimate for the pressure. 
\end{proof}

\begin{remark}
	The term $I_h^{33}$ is the precise reason why we needed to formulate the a~priori error estimate~\eqref{eq:fem:error_pres} in terms of the $s'$-norm and not in the stronger $p'$-norm.
\end{remark}

As an immediate consequence of Theorem \ref{prop:fem:error_pres} and Corollary \ref{cor:fem:error_vel}, in particular, observing that $\min\{1,2/p'\}\min\{1,p'/2\}=\min\{2/p',p'/2\}$, we obtain the following a~priori error estimates for the pressure with an explicit error decay rate.\enlargethispage{4mm}

\begin{cor}\label{cor:pressure_rates}
	Let the assumptions of Theorem~\ref{prop:fem:error} be satisfied. Then,~it~holds that
	\begin{align*}
		\|q_h - q\|_{s',\Omega}
		\leq c\, h^{\min \{ 2/p', p'/2 \}}\,,
	\end{align*}
	where $c > 0$ is a constant depending only on the characteristics of $\vec{S}$, $\delta^{-1}$, $\omega_0$,~$m$,~$k$, $\|\vec{v}\|_{\infty,\Omega}$, $\|\D{v}\|_{p,\Omega}$, $c_0$, and $\Omega$.
	If $p > 2$ and, in addition, $\vec{f} \in L^2(\Omega)$, then it holds that 
	\begin{align*}
		\|q_h - q\|_{s',\Omega}
		\leq c\, h\,,
	\end{align*}
	where $c > 0$ is a constant depending only on the characteristics of $\vec{S}$, $\delta^{-1}$, $\omega_0$,~$m$,~$k$, $\|\vec{v}\|_{\infty,\Omega}$, $\|\D{v}\|_{p,\Omega}$, $c_0$, and $\Omega$.
\end{cor}

The resulting a priori error estimate in Theorem~\ref{prop:fem:error_pres} is sub-linear if $p<2$ due to the step when switching from the modular to the norm (cf.~\eqref{prop:fem:error_pres.4}). We can circumvent this by computing the error in a norm with basis $\ell' \coloneqq  \min \{ 2,s' \} \leq 2$,~i.e.,~$\ell \coloneqq \max\{2,s\}$.
\begin{cor}\label{cor:alt_pressure}
	Let the assumptions of Theorem~\ref{prop:fem:error_pres} be satisfied and $p < 2$. Then, it holds that
	\begin{align*}
		\smash{\|q_h - q\|_{\ell',\Omega}
			\leq c \,h\,(\|\nabla \vec{F}(\D{v})\|_{2,\Omega}+\|\nabla q\|_{\ell',\Omega}) + c\,\|\vec{F}(\D{v}_h) - \vec{F}(\D{v})\|_{2,\Omega}\,,}
	\end{align*}
	where the constant $c > 0$ depends only on the characteristics of $\vec{S}$, $\delta^{-1}$, $\omega_0$,~$m$,~$k$,~and~$\Omega$.
\end{cor}

\begin{proof}
	Resorting to Hölder's inequality together with $\ell' \leq s'\leq p'$, i.e., $\ell\geq s\geq p$, the proof of Theorem~\ref{prop:fem:error_pres} can be followed step-by-step --except for the estimation of the term $I_h^1$--  with $s$ replaced by $\ell$ to obtain
	\begin{align}\label{cor:alt_pressure.1}
		\begin{aligned}
			\|q_h - q\|_{\ell',\Omega}
			&\leq c\, \sup_{\vec{z}_h \in V_h\colon \|\vec{z}_h\|_{1,\ell} \leq 1} {(\vec{S}(\D{v}_h) - \vec{S}(\D{v}), \D{z}_h)_\Omega} \\&\quad+ c \,h\,(\|\nabla \vec{F}(\D{v})\|_{2,\Omega}+\|\nabla q\|_{\ell',\Omega}) + c\,\|\vec{F}(\D{v}_h) - \vec{F}(\D{v})\|_{2,\Omega}\,.
		\end{aligned}
	\end{align}
	We note that, due to $p<2$  and $\delta>0$, for every $t\ge 0$ and a.e.\ $x\in \Omega$, it holds that 
	\begin{align}
		\label{eq:elementary}
	(\phi_{\vert \D{v}(x)\vert })\d (t) \sim ((\delta+\vert \D{v}(x)\vert)^{p-1}+t)^{p'-2}t^2\ge \delta^{2-p}t^2\,.
	\end{align}
	As a result, using Hölder's inequality together with $\ell' \leq 2$, \eqref{eq:elementary}, and \eqref{eq:growth_S}, we~arrive~at
	\begin{align}\label{cor:alt_pressure.2}
		\begin{aligned}
		\|\vec{S}(\D{v}_h) - \vec{S}(\D{v})\|_{\ell',\Omega}^2
	  &\leq c\, \|\vec{S}(\D{v}_h) - \vec{S}(\D{v})\|_{2,\Omega}^2
	\\&	\leq c\, \rho_{(\phi_{\abs{\D{v}}})\d,\Omega} (\vec{S}(\D{v}_h) - \vec{S}(\D{v}))
	\\&	\leq c\,\|\vec{F}(\D{v}_h) - \vec{F}(\D{v})\|_{2,\Omega}^2 \,.
\end{aligned}
	\end{align}
	Eventually, using Hölder's inequality together with \eqref{cor:alt_pressure.2} in \eqref{cor:alt_pressure.1}, we conclude that
	the claimed a priori error estimate for the pressure.
\end{proof}

Aided by Corollary \ref{cor:fem:error_vel},  using that for $p<2$ and $\delta>0$, due to \eqref{eq:elementary},~it~holds~that $\| \nabla q\|_{\ell',\Omega}^2\leq c\,\| \nabla q\|_{2,\Omega}^2\leq c\,\rho_{(\varphi_{\vert \D{v}\vert })^*,\Omega}(\nabla q)$,
from Corollary \ref{cor:alt_pressure}, we~obtain

\begin{cor}\label{cor:alt_pressure2}
	Let the assumptions of Theorem~\ref{prop:fem:error_pres} be satisfied and $p < 2$. Then, it holds that
	\begin{align*}
		\smash{\|q_h - q\|_{\ell',\Omega}
			\leq c \,h\,(\|\nabla \vec{F}(\D{v})\|_{2,\Omega}+\rho_{(\varphi_{\vert \D{v}\vert })^*,\Omega}(\nabla q)^{1/2})\,,}
	\end{align*}
	where $c > 0$ is a constant depending only on the characteristics of $\vec{S}$, $\delta^{-1}$, $\omega_0$,~$m$,~$k$, $\|\vec{v}\|_{\infty,\Omega}$, $\|\D{v}\|_{p,\Omega}$, $c_0$, and $\Omega$.\enlargethispage{20mm}
\end{cor}


\begin{remark}
	All the results in this section can easily be transferred to the corresponding generalized Stokes system. Then, due to the absence of the convective~term, we obtain slightly stronger results: we may consider  $p \in (1, \infty)$ and the smallness assumption~\eqref{eq:fem:error_smallness} is not necessary. Furthermore, we may resort to $s \hspace*{-0.1em}\coloneqq \hspace*{-0.1em} p$~and~${\ell\hspace*{-0.1em} \coloneqq\hspace*{-0.1em} 2}$~if~$p\hspace*{-0.1em} <\hspace*{-0.1em} 2$.\vspace*{-2.5mm}
\end{remark}

\begin{figure}[H] \label{fig:fem:pressure_rates}
\includegraphics[width=13cm]{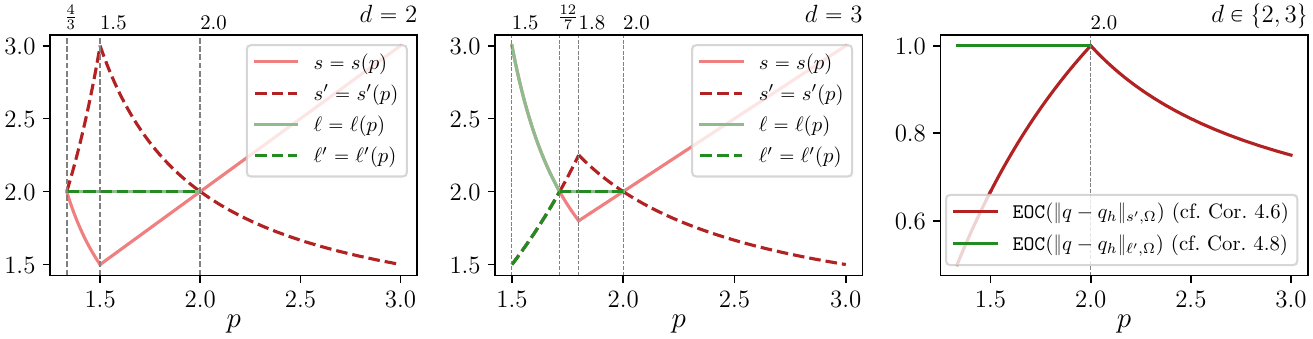}\vspace*{-2.5mm}
	\caption{LEFT: Plots of $s,s',\ell,\ell'\in (1,\infty)$ as  functions of $p\in (\frac{4}{3}(=\frac{2d}{d+1}),3)$ for $d=2$; MIDDLE: Plots of  $s,s',\ell,\ell'\in (1,\infty)$ as functions of $p\in (\frac{9}{4}(=\frac{2d}{d+1}),3)$~for~$d=3$;
		RIGHT: Plots of $\texttt{EOC}(\|q_h - q\|_{s',\Omega})$ (cf. Corollary \ref{cor:pressure_rates} and \eqref{eoc}) and $\texttt{EOC}(\|q_h - q\|_{\ell',\Omega})$ (cf. Corollary \ref{cor:alt_pressure2} and \eqref{eoc}) as functions of  $p\in (\frac{4}{3},3)$ for $d\in \{2,3\}$.}
\end{figure}

%% file: section4.tex
\newpage
\section{Numerical experiments} \label{sec:fem:num_experiments}

In this section, we review the theoretical findings~of Section \ref{sec:error} via numerical experiments.

\subsection{Implementation details}

All experiments were conducted deploying the finite element software package \texttt{FEniCS} (version 2019.1.0), cf.\ \cite{LW10}. 
In~the~numerical~experiments, 
we deploy the MINI element (cf.\ Remark \ref{FEM.V}(i)) and the Taylor--Hood element (cf.\ Remark \ref{FEM.V}(ii)), each in the two-dimensional case, i.e., $d=2$:\enlargethispage{3mm}

\textit{$\bullet$ Fortin interpolation operator for the MINI element.} Appealing to \cite[Sec. VI.4]{BBDR12_fem} or  \cite[Appx. A.1]{BBDR12_fem}, a Fortin interpolation operator $\Pi_h^X\colon X\to X_h$ for the MINI element, for every $\vec{z}\in X$, is defined by 
\begin{align*}
	\Pi_h^X\vec{z}\coloneqq 	\Pi_h^{\textit{sz},1}\vec{z}-\Pi_h^{\textit{corr}}(\vec{z}-\Pi_h^{\textit{sz},1}\vec{z})\quad\text{ in }X_h\,,
\end{align*}
where $	\Pi_h^{\textit{sz},1}\colon \hspace*{-0.1em}X\hspace*{-0.1em}\to\hspace*{-0.1em} \mathbb{P}^1(\mathcal{T}_h)\cap X$ is the \textit{first-order Scott--Zhang quasi-interpolation~\mbox{operator}} (cf.\ \cite{zhang-scott}) and the local correction mapping $	\Pi_h^{\textit{corr}}\colon X\to X_h$ enforcing the preservation of the discrete divergence in the $Y_h^*$-sense (cf.\ Assumption \ref{assum:fem:projection_x}(iii)), for every $\vec{z}\in X$, is defined by 
\begin{align*}
	\Pi_h^{\textit{corr}}\vec{z}\coloneqq  \sum_{K\in \mathcal{T}_h}{\smash{\frac{ \langle \vec{z}-\Pi_h^X\vec{z}\rangle_K}{\langle b_K\rangle_K}}b_K}\quad\text{ in }X_h\,,
\end{align*}
where $b_K\in \mathbb{P}^3(K)\cap W^{1,1}_0(K)\subseteq X_h$ for all $K\in \mathcal{T}_h$ denotes an element~bubble~function. Since for every $\vec{z}\in X$, by construction, it holds that $\Pi_h^{\textit{corr}}\vec{z}=\vec{0}$ a.e.\ on $\partial\Omega$ and, thus, $\Pi_h^X\vec{z}=	\Pi_h^{\textit{sz},1}\vec{z}$ a.e.\ on $\partial\Omega$, 
it is not necessary to implement the Fortin interpolation operator $\Pi_h^X\colon X\to X_h$, but only the first-order Scott--Zhang quasi-interpolation~operator $	\Pi_h^{\textit{sz},1}\colon X\to \mathbb{P}^1(\mathcal{T}_h)\cap X$, which is less computationally costly.

\textit{$\bullet$ Fortin interpolation operator for the Taylor--Hood element.} Appealing to \cite{DST22}, 
a Fortin interpolation operator $\Pi_h^X\hspace*{-0.15em}\colon\hspace*{-0.15em} X\hspace*{-0.2em}\to\hspace*{-0.2em} X_h$ for the Taylor--Hood element,~for~\mbox{every}~${\vec{z}\hspace*{-0.15em}\in \hspace*{-0.15em} X}$, is defined by 
\begin{align*}
	\Pi_h^X\vec{z}\coloneqq 	\Pi_h^{\textit{sz},2}\vec{z}-\Pi_h^{\textit{corr}}(\vec{z}-\Pi_h^{\textit{sz},2}\vec{z})\quad X_h\,,
\end{align*}
where $	\Pi_h^{\textit{sz},2}\colon \hspace*{-0.1em}X\hspace*{-0.1em}\to\hspace*{-0.1em} \mathbb{P}^1(\mathcal{T}_h)\cap X$ is the \textit{second-order Scott--Zhang quasi-interpolation~opera-tor} (cf.\ \cite{zhang-scott}) and the local correction mapping  $	\Pi_h^{\textit{corr}}\colon \hspace*{-0.1em}X\hspace*{-0.1em}\to\hspace*{-0.1em} X_h$ enforcing~the~\mbox{preservation} of the discrete divergence in the $Y_h^*$-sense (cf.\ Assumption \ref{assum:fem:projection_x}(iii)),~for~every~$\vec{z}\in X$, is defined by 
\begin{align*}
	\Pi_h^{\textit{corr}}\vec{z}\coloneqq  \sum_{F\in \Gamma_h\,;\;
		\nu_1,\nu_2\in \mathcal{N}_h\,\colon F=\textup{conv}\,\{\nu_1,\nu_2\}}{((\div(\varphi^1_{\nu_1}\vec{z}),\varphi^1_{\nu_2})_\Omega -(\div(\varphi^1_{\nu_2}\vec{z}),\varphi^1_{\nu_1})_\Omega)\,\boldsymbol{\psi}_F}\,,
\end{align*}
where $\Gamma_h$ is the set of facets of $\mathcal{T}_h$, $\mathcal{N}_h$ the set of vertices of $\mathcal{T}_h$, $(\varphi_\nu^1)_{\nu\in \mathcal{N}_h}\subseteq \mathbb{P}^1(\mathcal{T}_h)\cap X$  nodal \hspace*{-0.1mm}basis \hspace*{-0.1mm}functions, \hspace*{-0.1mm}and \hspace*{-0.1mm}$(\boldsymbol{\psi}_F)_{F\in \Gamma_h}\hspace*{-0.2em}\subseteq\hspace*{-0.2em} X_h$ \hspace*{-0.1mm}modified \hspace*{-0.1mm}tangential~\hspace*{-0.1mm}\mbox{bubble}~\hspace*{-0.1mm}\mbox{functions}~\hspace*{-0.1mm}(cf.~\hspace*{-0.1mm}\cite{DST22}). Since, again, for every $\vec{z}\in X$, by construction, it holds that $\Pi_h^{\textit{corr}}\vec{z}=\vec{0}$ a.e.\ on $\partial\Omega$ and, thus, $\Pi_h^X\vec{z}=	\Pi_h^{\textit{sz},1}\vec{z}$ a.e.\ on $\partial\Omega$,
it is not necessary to implement the Fortin interpolation operator $\Pi_h^X\colon X\to X_h$, but only the second-order Scott--Zhang quasi-interpolation operator $	\Pi_h^{\textit{sz},2}\colon X\to \mathbb{P}^1(\mathcal{T}_h)\cap X$.

We approximate the discrete solution $(\vec{v}_h,q_h)^{\top}\in X_h\times Q_h$ of the non-linear saddle point problem (i.e., Problem (\hyperlink{Ph}{P$_h$})) using the Newton solver from \mbox{\texttt{PETSc}}~(version~3.17.3), cf.~\cite{LW10}, with absolute tolerance~of $\tau_{abs}= 1\textrm{e}{-}8$ and relative tolerance~of~${\tau_{rel}=1\textrm{e}{-}10}$. The linear system emerging in each Newton iteration is solved using a sparse direct solver from \texttt{MUMPS} (version~5.5.0),~cf.~\cite{mumps}.  In the implementation, the uniqueness of the pressure is enforced via adding a zero mean condition.

\subsection{Experimental setup}

We employ Problem (\hyperlink{Qh}{Q$_h$}) (or~equivalently~\mbox{Problem} (\hyperlink{Ph}{P$_h$}))  to approximate the system~\eqref{eq:fem:main_problem1}--\eqref{eq:fem:main_problem3} with  $\vec{S}\colon  \mathbb{R}^{2\times 2}\to\mathbb{R}^{2\times 2}_{\mathrm{sym}}$, where ${\Omega=(0,1)^2}$, for  every $\vec{A}\in\mathbb{R}^{2\times 2}$~defined~by
\begin{align*}
	\vec{S}(\vec{A}) \coloneqq \nu\,(\delta+\vert \vec{A}^{\textup{sym}}\vert)^{p-2}\vec{A}^{\textup{sym}}\,,
\end{align*}  
where  $\delta\coloneqq1\textrm{e}{-}5$, $ p\in (4/3, 3.0)$, $\nu \coloneqq 0.1$ if $p\ge 2$, and $\nu\coloneqq  100$ if $p< 2$.

As manufactured solutions serve the vector field $\vec{v}\in X$ and the function $q \in Q^s$, 
defined by\enlargethispage{7mm}
\begin{align}
	\vec{v}\coloneqq\tfrac{1}{10}\vert \cdot\vert^\beta  \textrm{id}_{\mathbb{R}^2}\,, \qquad q\vcentcolon = \vert \cdot\vert^{\gamma}-\langle\,\vert \cdot\vert^{\gamma}\,\rangle_\Omega\,,
\end{align}
i.e., we choose the right-hand side $\vec{f}\in L^{p'}(\Omega)$, the divergence  $g_1\in L^s(\Omega)$ and boundary data $\vec{g}_2\in W^{\smash{1,1-\frac{1}{s}}}(\partial\Omega)$ accordingly.

Concerning the regularity of the velocity vector field, we choose
$\beta =  0.01$, 
which just  yields
that $\vec{F}(\vec{D}\vec{v})\in W^{1,2}(\Omega)$. 
Concerning the regularity of the pressure,~we~consider the two different cases covered by Corollary~\ref{cor:pressure_rates}: 

\begin{itemize}[leftmargin=!,labelwidth=\widthof{(\hypertarget{case_2}{Case 2})}]
	\item[\textit{(\hypertarget{case_1}{Case 1})}] We choose $\gamma= 1-\smash{\frac{2}{p'}}+0.01$, 
	which just yields that $q \in
	W^{1,p'}(\Omega)\cap L^{p'}_0(\Omega)$. 
	\item[\textit{(\hypertarget{case_2}{Case 2})}] We choose $\gamma= \beta \frac{p-2}{2}+0.01$,  
	which just yields that $(\delta+\abs{\vec{D}\vec{v}})^{\frac{2-p}{2}}\vert \nabla  q\vert  \in
	L^2(\Omega)$. 
\end{itemize}

We construct an initial triangulation $\mathcal
T_{h_0}$, where $h_0=1$, by subdividing the domain $\Omega=(0,1)^2$ along its diagonals into four triangles  with different orientations.  Then, finer triangulations~$\mathcal T_{h_i}$, $0=1,\dots,7$, where $h_{i+1}=\frac{h_i}{2}$ for all $i=0,\dots,11$, are 
obtained by
regular subdivision of the previous grid: each \mbox{triangle} is subdivided
into four equal triangles by connecting the midpoints of the edges, i.e., applying the red-refinement rule (cf.\ \cite[Def.~4.8(i)]{Ba16}). 

As  \hspace*{-0.1mm}estimation  \hspace*{-0.1mm}of  \hspace*{-0.1mm}the  \hspace*{-0.1mm}convergence  \hspace*{-0.1mm}rates,   \hspace*{-0.1mm}the  \hspace*{-0.1mm}experimental  \hspace*{-0.1mm}order  \hspace*{-0.1mm}of  \hspace*{-0.1mm}convergence~(EOC)
\begin{align}
	\texttt{EOC}_i(e_i)\coloneqq\frac{\log(e_i/e_{i-1})}{\log(h_i/h_{i-1})}\,, \quad i=1,\dots,7\,,\label{eoc}
\end{align}
where for every $i= 1,\dots,7$, we denote by $e_i$ a general error quantity.

\subsection{Quasi-optimality of  the error decay rates~in~the~Corollaries~\ref{cor:fem:error_vel},~\ref{cor:pressure_rates}}

For $p \in  \{4/3, 1.5, 1.75, 2.0, 2.25, 2.5,2.75,3.0\}$, the Mini and the Taylor--Hood element, in Case~\hyperlink{case_1}{1} and Case~\hyperlink{case_2}{2},  and a series of triangulations $\mathcal{T}_{h_i}$,~$i = 0, \dots, 7$, obtained~by~red-refinement as described above, we compute the corresponding discrete solutions $(\vec{v}_{h_i},q_{h_i})^\top\in X_{h_i}\times Q_{h_i}$, $i=0,\dots,7$, of Problem  (\hyperlink{Qh}{Q$_h$}), the error quantities 
\begin{align*}
	\left.\begin{aligned}
		e_{\vec{v},i}&\coloneqq\|\vec{F}(\vec{D}\vec{v}_{h_i})-\vec{F}(\vec{D}\vec{v})\|_{2,\Omega}\,,\\
		e_{q,i}^s&\coloneqq \|q_{h_i}-q\|_{s',\Omega}\,,
	\end{aligned}\quad\right\}\quad i=0,\dots,7\,,
\end{align*}
and 
the corresponding EOCs presented in Table \ref{tab:1}, Table \ref{tab:2}, Table \ref{tab:3}, and Table~\ref{tab:4}, respectively:\enlargethispage{5mm}

For \hspace*{-0.1mm}both 
\hspace*{-0.1mm}the \hspace*{-0.1mm}Mini  \hspace*{-0.1mm}and \hspace*{-0.1mm}the \hspace*{-0.1mm}Taylor--Hood \hspace*{-0.1mm}element~\hspace*{-0.1mm}as~\hspace*{-0.1mm}well~\hspace*{-0.1mm}as~\hspace*{-0.1mm}in~\hspace*{-0.1mm}Case~\hspace*{-0.1mm}\hyperlink{case_1}{1}~\hspace*{-0.1mm}and~\hspace*{-0.1mm}Case~\hspace*{-0.1mm}\hyperlink{case_2}{2}, for the velocity errors,  we report
the expected convergence rate of about $ \texttt{EOC}_i(e_{\vec{v},i}) \approx\min\{1,p'/2\}$, $i=1,\dots,7$, (in  Case \hyperlink{case_1}{1}) and $\texttt{EOC}_i(e_{\vec{v},i}) \approx 1$, $i=1,\dots,7$, (in  Case \hyperlink{case_2}{2}), while for the pressure errors, only in Case \hyperlink{case_1}{1} and if $p \in  \{1.5, 1.75, 2.0\}$, we observe 
the expected convergence rate of about $\texttt{EOC}_i(e_{q,i}^s) \approx p'/2$, $i=1,\dots,7$. In both the Mini  and the Taylor--Hood element as well as in Case \hyperlink{case_1}{1} and Case \hyperlink{case_2}{2}, for the pressure errors, in the case $p>2$, we report an increased 
convergence rate of about $\texttt{EOC}_i(e_{q,i}^s) \approx 1$, ${i=1,\dots,7}$, (in Case \hyperlink{case_1}{1}) and $\texttt{EOC}_i(e_{q,i}^s) \approx 2/p'$ , $i=1,\dots,7$, (in  Case \hyperlink{case_2}{2}).~Note~that~similar~results have  also been reported in \cite{BBDR12_fem,KR23_ldg3}. Putting everything~together,~in~any~case, 
we can confirm the quasi-optimality of the a priori error estimates for the velocity vector field derived in Corollary \ref{cor:fem:error_vel}, while only in the case
 $p\in [1.5,2]$, we can confirm the quasi-optimality of the a priori error estimates for the pressure derived in Corollary~\ref{cor:pressure_rates}.

\begin{table}[H]
	\setlength\tabcolsep{1.9pt}
	\centering
	\begin{tabular}{c |c|c|c|c|c|c|c|c|c|c|c|c|c|} \cmidrule(){2-13}
		
		& \multicolumn{8}{c||}{\cellcolor{lightgray}Case \hyperlink{case_1}{1}}   & \multicolumn{4}{c|}{\cellcolor{lightgray}Case \hyperlink{case_2}{2}}\\ 
		\hline 
		
		\multicolumn{1}{|c||}{\cellcolor{lightgray}\diagbox[height=1.1\line,width=0.11\dimexpr\linewidth]{\vspace{-0.6mm}$i$}{\\[-5mm] $p$}}
		& \cellcolor{lightgray}$4/3$ & \cellcolor{lightgray}1.5 & \cellcolor{lightgray}1.75  & \cellcolor{lightgray}2.0  &  \cellcolor{lightgray}2.25 &
		\cellcolor{lightgray}2.5  & \cellcolor{lightgray}2.75 & \multicolumn{1}{c||}{\cellcolor{lightgray}3.0} &  \multicolumn{1}{c|}{\cellcolor{lightgray}2.25}  & \cellcolor{lightgray}2.5  & \cellcolor{lightgray}2.75 & \cellcolor{lightgray}3.0  \\ \hline\hline
		\multicolumn{1}{|c||}{\cellcolor{lightgray}$1$}            & 0.994 & 0.993 & 0.991 & 0.863 & 0.767 & 0.695 & 0.641 & \multicolumn{1}{c||}{0.601} & \multicolumn{1}{c|}{0.853} & 0.845 & 0.834 & 0.817 \\ \hline
		\multicolumn{1}{|c||}{\cellcolor{lightgray}$2$}            & 0.926 & 0.927 & 0.928 & 0.975 & 0.876 & 0.804 & 0.757 & \multicolumn{1}{c||}{0.722} & \multicolumn{1}{c|}{0.977} & 0.976 & 0.972 & 0.966 \\ \hline
		\multicolumn{1}{|c||}{\cellcolor{lightgray}$3$}            & 0.934 & 0.935 & 0.935 & 0.996 & 0.894 & 0.823 & 0.778 & \multicolumn{1}{c||}{0.745} & \multicolumn{1}{c|}{1.000} & 1.001 & 1.001 & 0.999 \\ \hline
		\multicolumn{1}{|c||}{\cellcolor{lightgray}$4$}            & 0.940 & 0.940 & 0.940 & 1.001 & 0.898 & 0.827 & 0.783 & \multicolumn{1}{c||}{0.750} & \multicolumn{1}{c|}{1.005} & 1.007 & 1.007 & 1.007 \\ \hline
		\multicolumn{1}{|c||}{\cellcolor{lightgray}$5$}            & 0.945 & 0.945 & 0.945 & 1.002 & 0.899 & 0.829 & 0.785 & \multicolumn{1}{c||}{0.752} & \multicolumn{1}{c|}{1.007} & 1.008 & 1.009 & 1.009 \\ \hline
		\multicolumn{1}{|c||}{\cellcolor{lightgray}$6$}            & 0.949 & 0.949 & 0.949 & 1.002 & 0.899 & 0.831 & 0.787 & \multicolumn{1}{c||}{0.753} & \multicolumn{1}{c|}{1.007} & 1.009 & 1.009 & 1.010 \\ \hline
		\multicolumn{1}{|c||}{\cellcolor{lightgray}$7$}            & 0.952 & 0.953 & 0.953 & 1.002 & 0.900 & 0.832 & 0.788 & \multicolumn{1}{c||}{0.753} & \multicolumn{1}{c|}{1.007} & 1.009 & 1.010 & 1.010  \\ \hline\hline
		\multicolumn{1}{|c||}{\cellcolor{lightgray}\small theory}  & 1.000 & 1.000 & 1.000 & 1.000 & 0.900 & 0.833 & 0.786 & \multicolumn{1}{c||}{0.750} & \multicolumn{1}{c|}{1.000} & 1.000 & 1.000 & 1.000 \\ \hline
	\end{tabular}\vspace{-2mm}
	\caption{Experimental order of convergence (MINI): $\texttt{EOC}_i(e_{\vec{v},i})$,~${i=1,\dots,7}$.}\label{tab:1}
\end{table}\enlargethispage{10mm}\vspace{-7.5mm}

\begin{table}[H]
	\setlength\tabcolsep{1.9pt}
	\centering
	\begin{tabular}{c |c|c|c|c|c|c|c|c|c|c|c|c|c|} \cmidrule(){2-13}
		
		& \multicolumn{8}{c||}{\cellcolor{lightgray}Case \hyperlink{case_1}{1}}   & \multicolumn{4}{c|}{\cellcolor{lightgray}Case \hyperlink{case_2}{2}}\\ 
		\hline 
		
		\multicolumn{1}{|c||}{\cellcolor{lightgray}\diagbox[height=1.1\line,width=0.11\dimexpr\linewidth]{\vspace{-0.6mm}$i$}{\\[-5mm] $p$}}
		& \cellcolor{lightgray}$4/3$ & \cellcolor{lightgray}1.5 & \cellcolor{lightgray}1.75  & \cellcolor{lightgray}2.0  &  \cellcolor{lightgray}2.25 &
		\cellcolor{lightgray}2.5  & \cellcolor{lightgray}2.75 & \multicolumn{1}{c||}{\cellcolor{lightgray}3.0} &  \multicolumn{1}{c|}{\cellcolor{lightgray}2.25}  & \cellcolor{lightgray}2.5  & \cellcolor{lightgray}2.75 & \cellcolor{lightgray}3.0  \\ \hline\hline
		\multicolumn{1}{|c||}{\cellcolor{lightgray}$1$}            & 1.494 & 1.010 & 1.313 & 0.980 & 0.972 & 0.955 & 0.943 & \multicolumn{1}{c||}{0.934} & \multicolumn{1}{c|}{1.060} & 1.114 & 1.153 & 1.182 \\ \hline
		\multicolumn{1}{|c||}{\cellcolor{lightgray}$2$}            & 0.989 & 0.640 & 0.850 & 1.010 & 1.006 & 1.003 & 1.001 & \multicolumn{1}{c||}{0.999} & \multicolumn{1}{c|}{1.116} & 1.197 & 1.260 & 1.310 \\ \hline
		\multicolumn{1}{|c||}{\cellcolor{lightgray}$3$}            & 0.999 & 0.658 & 0.858 & 1.009 & 1.008 & 1.007 & 1.006 & \multicolumn{1}{c||}{1.006} & \multicolumn{1}{c|}{1.119} & 1.204 & 1.271 & 1.325 \\ \hline
		\multicolumn{1}{|c||}{\cellcolor{lightgray}$4$}            & 1.002 & 0.664 & 0.861 & 1.010 & 1.010 & 1.009 & 1.009 & \multicolumn{1}{c||}{1.009} & \multicolumn{1}{c|}{1.121} & 1.209 & 1.279 & 1.336 \\ \hline
		\multicolumn{1}{|c||}{\cellcolor{lightgray}$5$}            & 1.002 & 0.666 & 0.862 & 1.010 & 1.010 & 1.010 & 1.010 & \multicolumn{1}{c||}{1.010} & \multicolumn{1}{c|}{1.122} & 1.211 & 1.283 & 1.342 \\ \hline
		\multicolumn{1}{|c||}{\cellcolor{lightgray}$6$}            & 1.002 & 0.668 & 0.862 & 1.010 & 1.010 & 1.010 & 1.010 & \multicolumn{1}{c||}{1.011} & \multicolumn{1}{c|}{1.122} & 1.212 & 1.285 & 1.345 \\ \hline
		\multicolumn{1}{|c||}{\cellcolor{lightgray}$7$}            & 1.003 & 0.669 & 0.862 & 1.010 & 1.010 & 1.010 & 1.011 & \multicolumn{1}{c||}{1.011} & \multicolumn{1}{c|}{1.122} & 1.212 & 1.286 & 1.347 \\ \hline\hline
		\multicolumn{1}{|c||}{\cellcolor{lightgray}\small theory}  & 0.500 & 0.667 & 0.875 & 1.000 & 0.900 & 0.833 & 0.786 & \multicolumn{1}{c||}{0.750} & \multicolumn{1}{c|}{1.000} & 1.000 & 1.000 & 1.000 \\ \hline
	\end{tabular}\vspace{-2mm}
	\caption{Experimental order of convergence (MINI): $\texttt{EOC}_i(e_{q,i}^s)$,~${i=1,\dots,7}$.}\label{tab:2}
\end{table}\vspace{-7.5mm}

\begin{table}[H]
	\setlength\tabcolsep{1.9pt}
	\centering
	\begin{tabular}{c |c|c|c|c|c|c|c|c|c|c|c|c|c|} \cmidrule(){2-13}
		
		& \multicolumn{8}{c||}{\cellcolor{lightgray}Case \hyperlink{case_1}{1}}   & \multicolumn{4}{c|}{\cellcolor{lightgray}Case \hyperlink{case_2}{2}}\\ 
		\hline 
		
		\multicolumn{1}{|c||}{\cellcolor{lightgray}\diagbox[height=1.1\line,width=0.11\dimexpr\linewidth]{\vspace{-0.6mm}$i$}{\\[-5mm] $p$}}
		& \cellcolor{lightgray}$4/3$ & \cellcolor{lightgray}1.5 & \cellcolor{lightgray}1.75  & \cellcolor{lightgray}2.0  &  \cellcolor{lightgray}2.25 &
		\cellcolor{lightgray}2.5  & \cellcolor{lightgray}2.75 & \multicolumn{1}{c||}{\cellcolor{lightgray}3.0} &  \multicolumn{1}{c|}{\cellcolor{lightgray}2.25}  & \cellcolor{lightgray}2.5  & \cellcolor{lightgray}2.75 & \cellcolor{lightgray}3.0  \\ \hline\hline
		\multicolumn{1}{|c||}{\cellcolor{lightgray}$1$}            & 0.990 & 0.991 & 0.994 & 0.793 & 0.733 & 0.683 & 0.644 & \multicolumn{1}{c||}{0.614} & \multicolumn{1}{c|}{0.789} & 0.782 & 0.769 & 0.744 \\ \hline
		\multicolumn{1}{|c||}{\cellcolor{lightgray}$2$}            & 1.007 & 1.008 & 1.007 & 0.982 & 0.880 & 0.812 & 0.766 & \multicolumn{1}{c||}{0.730} & \multicolumn{1}{c|}{0.981} & 0.978 & 0.974 & 0.966 \\ \hline
		\multicolumn{1}{|c||}{\cellcolor{lightgray}$3$}            & 1.009 & 1.010 & 1.010 & 1.005 & 0.897 & 0.828 & 0.783 & \multicolumn{1}{c||}{0.748} & \multicolumn{1}{c|}{1.004} & 1.003 & 1.002 & 1.000 \\ \hline
		\multicolumn{1}{|c||}{\cellcolor{lightgray}$4$}            & 1.008 & 1.010 & 1.010 & 1.009 & 0.900 & 0.832 & 0.786 & \multicolumn{1}{c||}{0.752} & \multicolumn{1}{c|}{1.009} & 1.008 & 1.008 & 1.008 \\ \hline
		\multicolumn{1}{|c||}{\cellcolor{lightgray}$5$}            & 1.008 & 1.009 & 1.010 & 1.010 & 0.901 & 0.833 & 0.788 & \multicolumn{1}{c||}{0.753} & \multicolumn{1}{c|}{1.010} & 1.010 & 1.009 & 1.009 \\ \hline
		\multicolumn{1}{|c||}{\cellcolor{lightgray}$6$}            & 1.007 & 1.009 & 1.010 & 1.010 & 0.901 & 0.834 & 0.789 & \multicolumn{1}{c||}{0.754} & \multicolumn{1}{c|}{1.010} & 1.010 & 1.010 & 1.010 \\ \hline
		\multicolumn{1}{|c||}{\cellcolor{lightgray}$7$}            & 1.007 & 1.008 & 1.010 & 1.010 & 0.901 & 0.835 & 0.789 & \multicolumn{1}{c||}{0.755} & \multicolumn{1}{c|}{1.010} & 1.010 & 1.010 & 1.010  \\ \hline\hline
		\multicolumn{1}{|c||}{\cellcolor{lightgray}\small theory}  & 1.000 & 1.000 & 1.000 & 1.000 & 0.900 & 0.833 & 0.786 & \multicolumn{1}{c||}{0.750} & \multicolumn{1}{c|}{1.000} & 1.000 & 1.000 & 1.000 \\ \hline
	\end{tabular}\vspace{-2mm}
	\caption{Experimental order of convergence (Taylor--Hood): $\texttt{EOC}_i(e_{\vec{v},i})$,~${i=1,\dots,7}$.}\label{tab:3}
\end{table}\vspace{-7.5mm}

\begin{table}[H]
	\setlength\tabcolsep{1.9pt}
	\centering
	\begin{tabular}{c |c|c|c|c|c|c|c|c|c|c|c|c|c|} \cmidrule(){2-13}
		
		& \multicolumn{8}{c||}{\cellcolor{lightgray}Case \hyperlink{case_1}{1}}   & \multicolumn{4}{c|}{\cellcolor{lightgray}Case \hyperlink{case_2}{2}}\\ 
		\hline 
		
		\multicolumn{1}{|c||}{\cellcolor{lightgray}\diagbox[height=1.1\line,width=0.11\dimexpr\linewidth]{\vspace{-0.6mm}$i$}{\\[-5mm] $p$}}
		& \cellcolor{lightgray}$4/3$ & \cellcolor{lightgray}1.5 & \cellcolor{lightgray}1.75  & \cellcolor{lightgray}2.0  &  \cellcolor{lightgray}2.25 &
		\cellcolor{lightgray}2.5  & \cellcolor{lightgray}2.75 & \multicolumn{1}{c||}{\cellcolor{lightgray}3.0} &  \multicolumn{1}{c|}{\cellcolor{lightgray}2.25}  & \cellcolor{lightgray}2.5  & \cellcolor{lightgray}2.75 & \cellcolor{lightgray}3.0  \\ \hline\hline
		\multicolumn{1}{|c||}{\cellcolor{lightgray}$1$}            & 1.255 & 1.011 & 0.987 & 0.965 & 0.967 & 0.967 & 0.968 & \multicolumn{1}{c||}{0.970} & \multicolumn{1}{c|}{1.057} & 1.121 & 1.170 & 1.211 \\ \hline
		\multicolumn{1}{|c||}{\cellcolor{lightgray}$2$}            & 1.256 & 0.961 & 0.990 & 1.005 & 1.004 & 1.002 & 0.999 & \multicolumn{1}{c||}{0.997} & \multicolumn{1}{c|}{1.111} & 1.191 & 1.255 & 1.308 \\ \hline
		\multicolumn{1}{|c||}{\cellcolor{lightgray}$3$}            & 1.171 & 0.890 & 0.985 & 1.009 & 1.008 & 1.008 & 1.007 & \multicolumn{1}{c||}{1.006} & \multicolumn{1}{c|}{1.119} & 1.204 & 1.272 & 1.327 \\ \hline
		\multicolumn{1}{|c||}{\cellcolor{lightgray}$4$}            & 1.101 & 0.822 & 0.981 & 1.010 & 1.010 & 1.009 & 1.009 & \multicolumn{1}{c||}{1.009} & \multicolumn{1}{c|}{1.121} & 1.209 & 1.280 & 1.337 \\ \hline
		\multicolumn{1}{|c||}{\cellcolor{lightgray}$5$}            & 1.051 & 0.770 & 0.974 & 1.010 & 1.010 & 1.010 & 1.010 & \multicolumn{1}{c||}{1.010} & \multicolumn{1}{c|}{1.122} & 1.211 & 1.283 & 1.343 \\ \hline
		\multicolumn{1}{|c||}{\cellcolor{lightgray}$6$}            & 1.023 & 0.739 & 0.967 & 1.010 & 1.010 & 1.010 & 1.010 & \multicolumn{1}{c||}{1.010} & \multicolumn{1}{c|}{1.122} & 1.212 & 1.285 & 1.345 \\ \hline
		\multicolumn{1}{|c||}{\cellcolor{lightgray}$7$}            & 1.008 & 0.720 & 0.958 & 1.010 & 1.010 & 1.010 & 1.010 & \multicolumn{1}{c||}{1.010} & \multicolumn{1}{c|}{1.122} & 1.212 & 1.286 & 1.347 \\ \hline\hline
		\multicolumn{1}{|c||}{\cellcolor{lightgray}\small theory}  & 0.500 & 0.667 & 0.875 & 1.000 & 0.900 & 0.833 & 0.786 & \multicolumn{1}{c||}{0.750} & \multicolumn{1}{c|}{1.000} & 1.000 & 1.000 & 1.000 \\ \hline
	\end{tabular}\vspace{-2mm}
	\caption{Experimental order of convergence (Taylor--Hood): $\texttt{EOC}_i(e_{q,i}^s)$,~${i=1,\dots,7}$.}\label{tab:4}
\end{table}\vspace{-7.5mm}

\subsection{Quasi-optimality of  the error decay rate in~Corollary~\ref{cor:alt_pressure2}}

For~$p \in  \{4/3, 1.4, 1.45, 1.5, 1.75, 2.0\}$, the Mini and the Taylor--Hood element, in Case~\hyperlink{case_1}{1},  and a series of triangulations $\mathcal{T}_{h_i}$,~$i = 0, \dots, 7$, obtained~by~red-refinement as described above, we compute the corresponding discrete solutions $(\vec{v}_{h_i},q_{h_i})^\top\in X_{h_i}\times Q_{h_i}$, $i=0,\dots,7$, of Problem  (\hyperlink{Qh}{Q$_h$}), the error quantities\vspace*{-1mm}
\begin{align*}
	\left.\begin{aligned}
		e_{q,i}^{\ell}&\coloneqq \|q_{h_i}-q\|_{\ell',\Omega}\,,\\
		e_{q,i}^s&\coloneqq \|q_{h_i}-q\|_{s',\Omega}\,,\\
		e_{q,i}^p&\coloneqq \|q_{h_i}-q\|_{p',\Omega}\,,
	\end{aligned}\quad\right\}\quad i=0,\dots,7\,,
\end{align*}
and 
the corresponding EOCs presented in Table \ref{tab:5} and Table  \ref{tab:6}, respectively: \enlargethispage{15mm}

For both the Mini and the Taylor--Hood  element,  for the errors $e_{q,i}^{\ell}$, $i=0,\dots,7$,~we report
the expected convergence rate of about $ \texttt{EOC}_i(e_{q,i}^{\ell}) \approx 1$, $i=1,\dots,7$, while for the errors $e_{q,i}^s$, $i=0,\dots,7$, 
we report a  convergence rate~of~about~${\texttt{EOC}_i(e_{q,i}^s) \approx 2/s'}$, ${i=1,\dots,7}$ and for the errors $e_{q,i}^p$, $i\hspace*{-0.1em}=\hspace*{-0.1em}0,\dots,7$, we report a convergence rate of about 
$\texttt{EOC}_i(e_{q,i}^p)\hspace*{-0.1em} \approx \hspace*{-0.1em}2/p'$, $i\hspace*{-0.1em}=\hspace*{-0.1em}1,\dots,7$. The~increased~convergence~rates (compared~to~Corollary~\ref{cor:pressure_rates}) may be explained  be the fact that, for this example, we have that $q\in L^\infty(\Omega)$, so that from Corollary \ref{cor:alt_pressure2} together with local inverse estimates (cf. \cite[Lem. 12.1]{EG21}) and the $L^\infty$-stability of $\Pi_h^Y$ (cf.\ \cite[Lem. 5.2]{BBDR12_fem}) and a standard argument, we can deduce that
$\sup_{h\in (0,1]}{\|q_h\|_{\infty,\Omega}}\hspace*{-0.1em}<\hspace*{-0.1em}\infty$ and, thus, with real interpolation and,~again,~Corollary~\ref{cor:alt_pressure2}~that\vspace*{-1mm}
\begin{align*}
	\|q_h-q\|_{s',\Omega}\leq \|q_h-q\|_{2,\Omega}^{2/s'}\|q_h-q\|_{\infty,\Omega}^{1-2/s'}\leq c\,h^{2/s'}\,,\\
		\|q_h-q\|_{p',\Omega}\leq \|q_h-q\|_{2,\Omega}^{2/p'}\|q_h-q\|_{\infty,\Omega}^{1-2/p'}\leq c\,h^{2/p'}\,.
\end{align*}
Putting everything together, in any case, 
we can confirm the quasi-optimality of the a priori error estimates for the pressure  derived in Corollary \ref{cor:alt_pressure2}. 

\begin{table}[H]
	\setlength\tabcolsep{1.8pt}
	\centering
	\begin{tabular}{c |c|c|c|c|c|c|c|c|c|c|c|c|} \cmidrule(){2-13}
		
		& \multicolumn{6}{c||}{\cellcolor{lightgray}$e_{q,i}^{\ell}$\phantom{$X^{X^X}_{X_{X_X}}$}}   & \multicolumn{3}{c||}{\cellcolor{lightgray}$e_{q,i}^s$} & \multicolumn{3}{c|}{\cellcolor{lightgray}$e_{q,i}^p$}\\ 
		\hline 
		
		\multicolumn{1}{|c||}{\cellcolor{lightgray}\diagbox[height=1.1\line,width=0.11\dimexpr\linewidth]{\vspace{-0.6mm}$i$}{\\[-5mm] $p$}}
		& \cellcolor{lightgray}$4/3$  & \cellcolor{lightgray}1.40  & \cellcolor{lightgray}1.45  & \cellcolor{lightgray}1.5 & \cellcolor{lightgray}1.75 & \multicolumn{1}{c||}{\cellcolor{lightgray}2.0} &
		\cellcolor{lightgray}$4/3$ & \cellcolor{lightgray}1.40  & \multicolumn{1}{c||}{\cellcolor{lightgray}1.45}   &
		\cellcolor{lightgray}$4/3$ & \cellcolor{lightgray}1.40  &\cellcolor{lightgray}1.45      \\ \hline\hline
		\multicolumn{1}{|c||}{\cellcolor{lightgray}$1$}            & 1.494 & 1.497 & 1.499 & 1.502 & 1.513 & \multicolumn{1}{c||}{1.513} & \multicolumn{1}{c|}{1.494} & 1.293 & \multicolumn{1}{c||}{1.147} & \multicolumn{1}{c|}{0.756} & 0.861 & 0.937 \\ \hline
		\multicolumn{1}{|c||}{\cellcolor{lightgray}$2$}            & 0.989 & 0.991 & 0.992 & 0.993 & 1.002 & \multicolumn{1}{c||}{1.009} & \multicolumn{1}{c|}{0.989} & 0.839 & \multicolumn{1}{c||}{0.736} & \multicolumn{1}{c|}{0.466} & 0.540 & 0.591 \\ \hline
		\multicolumn{1}{|c||}{\cellcolor{lightgray}$3$}            & 0.999 & 1.000 & 1.001 & 1.001 & 1.006 & \multicolumn{1}{c||}{1.012} & \multicolumn{1}{c|}{0.999} & 0.852 & \multicolumn{1}{c||}{0.751} & \multicolumn{1}{c|}{0.488} & 0.560 & 0.611 \\ \hline
		\multicolumn{1}{|c||}{\cellcolor{lightgray}$4$}            & 1.002 & 1.002 & 1.002 & 1.003 & 1.006 & \multicolumn{1}{c||}{1.011} & \multicolumn{1}{c|}{1.002} & 0.856 & \multicolumn{1}{c||}{0.757} & \multicolumn{1}{c|}{0.496} & 0.567 & 0.617 \\ \hline
		\multicolumn{1}{|c||}{\cellcolor{lightgray}$5$}            & 1.002 & 1.002 & 1.003 & 1.003 & 1.006 & \multicolumn{1}{c||}{1.011} & \multicolumn{1}{c|}{1.002} & 0.858 & \multicolumn{1}{c||}{0.759} & \multicolumn{1}{c|}{0.499} & 0.570 & 0.620 \\ \hline
		\multicolumn{1}{|c||}{\cellcolor{lightgray}$6$}            & 1.002 & 1.003 & 1.003 & 1.003 & 1.006 & \multicolumn{1}{c||}{1.010} & \multicolumn{1}{c|}{1.002} & 0.859 & \multicolumn{1}{c||}{0.760} & \multicolumn{1}{c|}{0.500} & 0.572 & 0.621 \\ \hline
		\multicolumn{1}{|c||}{\cellcolor{lightgray}$7$}            & 1.003 & 1.003 & 1.003 & 1.004 & 1.006 & \multicolumn{1}{c||}{1.010} & \multicolumn{1}{c|}{1.003} & 0.860 & \multicolumn{1}{c||}{0.761} & \multicolumn{1}{c|}{0.501} & 0.573 & 0.622 \\ \hline
		\multicolumn{1}{|c||}{\cellcolor{lightgray}\small theory}  & 1.000 & 1.000 & 1.000 & 1.000 & 1.000 & \multicolumn{1}{c||}{1.000} & \multicolumn{1}{c|}{0.500} & 0.571 & \multicolumn{1}{c||}{0.621} & \multicolumn{1}{c|}{ --- } &  ---  &  ---  \\ \hline
		\end{tabular}\vspace{-2mm}
	\caption{Experimental order of convergence (MINI): $\texttt{EOC}_i(e_{q,i}^{\textup{a}})$,~${i\hspace*{-0.1em}=\hspace*{-0.1em}1,\dots,7}$,~${\textup{a}\hspace*{-0.1em}\in\hspace*{-0.1em} \{\ell,s,p\}}$.}\label{tab:5}
\end{table}\vspace{-7.5mm}

\begin{table}[H]
	\setlength\tabcolsep{1.8pt}
	\centering
	\begin{tabular}{c |c|c|c|c|c|c|c|c|c|c|c|c|} \cmidrule(){2-13}
		
		& \multicolumn{6}{c||}{\cellcolor{lightgray}$e_{q,i}^{\ell}$\phantom{$X^{X^X}_{X_{X_X}}$}}   & \multicolumn{3}{c||}{\cellcolor{lightgray}$e_{q,i}^s$} & \multicolumn{3}{c|}{\cellcolor{lightgray}$e_{q,i}^p$}\\ 
		\hline 
		
		\multicolumn{1}{|c||}{\cellcolor{lightgray}\diagbox[height=1.1\line,width=0.11\dimexpr\linewidth]{\vspace{-0.6mm}$i$}{\\[-5mm] $p$}}
		& \cellcolor{lightgray}$4/3$  & \cellcolor{lightgray}1.40  & \cellcolor{lightgray}1.45  & \cellcolor{lightgray}1.5 & \cellcolor{lightgray}1.75 & \multicolumn{1}{c||}{\cellcolor{lightgray}2.0} &
		\cellcolor{lightgray}$4/3$ & \cellcolor{lightgray}1.40  & \multicolumn{1}{c||}{\cellcolor{lightgray}1.45}   &
		\cellcolor{lightgray}$4/3$ & \cellcolor{lightgray}1.40  &\cellcolor{lightgray}1.45      \\ \hline\hline
		\multicolumn{1}{|c||}{\cellcolor{lightgray}$1$}            & 1.255 & 1.241 & 1.226 & 1.207 & 1.096 & \multicolumn{1}{c||}{1.100} & \multicolumn{1}{c|}{1.255} & 1.184 & \multicolumn{1}{c||}{1.108} & \multicolumn{1}{c|}{1.000} & 1.023 & 1.019 \\ \hline
		\multicolumn{1}{|c||}{\cellcolor{lightgray}$2$}            & 1.256 & 1.262 & 1.256 & 1.243 & 1.121 & \multicolumn{1}{c||}{0.962} & \multicolumn{1}{c|}{1.256} & 1.128 & \multicolumn{1}{c||}{1.040} & \multicolumn{1}{c|}{0.686} & 0.849 & 0.923 \\ \hline
		\multicolumn{1}{|c||}{\cellcolor{lightgray}$3$}            & 1.171 & 1.199 & 1.210 & 1.210 & 1.118 & \multicolumn{1}{c||}{0.988} & \multicolumn{1}{c|}{1.171} & 1.039 & \multicolumn{1}{c||}{0.957} & \multicolumn{1}{c|}{0.594} & 0.718 & 0.816 \\ \hline
		\multicolumn{1}{|c||}{\cellcolor{lightgray}$4$}            & 1.101 & 1.138 & 1.160 & 1.173 & 1.116 & \multicolumn{1}{c||}{1.001} & \multicolumn{1}{c|}{1.101} & 0.972 & \multicolumn{1}{c||}{0.889} & \multicolumn{1}{c|}{0.561} & 0.662 & 0.741 \\ \hline
		\multicolumn{1}{|c||}{\cellcolor{lightgray}$5$}            & 1.051 & 1.084 & 1.110 & 1.131 & 1.111 & \multicolumn{1}{c||}{1.007} & \multicolumn{1}{c|}{1.051} & 0.927 & \multicolumn{1}{c||}{0.844} & \multicolumn{1}{c|}{0.529} & 0.633 & 0.700 \\ \hline
		\multicolumn{1}{|c||}{\cellcolor{lightgray}$6$}            & 1.023 & 1.046 & 1.069 & 1.092 & 1.105 & \multicolumn{1}{c||}{1.009} & \multicolumn{1}{c|}{1.023} & 0.894 & \multicolumn{1}{c||}{0.816} & \multicolumn{1}{c|}{0.505} & 0.607 & 0.677 \\ \hline
		\multicolumn{1}{|c||}{\cellcolor{lightgray}$7$}            & 1.008 & 1.023 & 1.040 & 1.060 & 1.097 & \multicolumn{1}{c||}{1.010} & \multicolumn{1}{c|}{1.008} & 0.874 & \multicolumn{1}{c||}{0.792} & \multicolumn{1}{c|}{0.492} & 0.586 & 0.656 \\ \hline
		\multicolumn{1}{|c||}{\cellcolor{lightgray}\small theory}  & 1.000 & 1.000 & 1.000 & 1.000 & 1.000 & \multicolumn{1}{c||}{1.000} & \multicolumn{1}{c|}{0.500} & 0.571 & \multicolumn{1}{c||}{0.621} & \multicolumn{1}{c|}{ --- } &  ---  &  ---  \\ \hline
		\end{tabular}\vspace{-2mm}
	\caption{Experimental order of convergence (Taylor--Hood): $\texttt{EOC}_i(e_{q,i}^{\textup{a}})$,~${i\hspace*{-0.1em}=\hspace*{-0.1em}1,\dots,7}$, ${\textup{a}\hspace*{-0.1em}\in\hspace*{-0.1em} \{\ell,s,p\}}$.}\label{tab:6}
\end{table}\vspace{-7.5mm}